\documentclass[invmat,final]{svjour}
\usepackage{amsmath,amsfonts,amscd,amssymb,graphicx}
\newcommand{\ip}{\int_\mathbb{R}}

\newcommand{\C}{\mathbb{C}}

\def\bb1{{1\!\!1}}

\def\bU{{\hat{u}}}
\def\bV{{\hat{v}}}
\def\R{\Re e}
\def\I{\Im m}

\begin{document}
\title{Stability of isentropic viscous shock profiles in the high-Mach number limit}
\titlerunning{Stability of isentropic viscous shock profiles}
\author{Jeffrey Humpherys\inst{1} \and Olivier Laffite\inst{2} \and Kevin Zumbrun\inst{3}
}                     
%
%
\institute{Department of Mathematics, Brigham Young University, Provo, UT 84602 \and LAGA, Institut Galilee, Universite Paris 13, 93 430 Villetaneuse and CEA Saclay, DM2S/DIR, 91 191 Gif sur Yvette Cedex \and Department of Mathematics, Indiana University, Bloomington, IN 47402}
\date{Received: date / Revised version: date}
%
\maketitle
\begin{abstract}
By a combination of asymptotic ODE estimates and numerical Evans
function calculations,
we establish 
stability of viscous shock solutions of
the isentropic compressible Navier--Stokes equations
with $\gamma$-law pressure
(i) in the limit as Mach number $M$ goes to infinity, 
for any $\gamma\ge 1$ (proved analytically),
and (ii) for $M\ge 2,500$, $\gamma\in [1,2.5]$
(demonstrated numerically).
This builds on and completes earlier studies by
Matsumura--Nishihara and
Barker--Humpherys--Rudd--Zumbrun establishing stability for 
low and intermediate Mach numbers, respectively,
indicating unconditional stability, independent of shock amplitude,
of viscous shock waves
for $\gamma$-law gas dynamics in the range $\gamma \in [1,2.5]$.
Other $\gamma$-values may be treated similarly, but have not
been checked numerically.
The main idea is to establish convergence of the Evans
function in the high-Mach number limit to that of a pressureless,
or ``infinitely compressible'', 
gas with additional upstream boundary condition determined
by a boundary-layer analysis.
Recall that low-Mach number behavior is incompressible.
\end{abstract}
%
\section{Introduction}
\label{int}

The isentropic compressible Navier-Stokes equations in one 
spatial dimension expressed in Lagrangian coordinates take the form
\begin{equation}
\begin{split}
\label{psystem}
v_{t}-u_{x} &=0, \\
u_{t}+p(v)_{x} &= \left(\frac{u_x}{v}\right)_x,
\end{split}
\end{equation}
where $v$ is specific volume, $u$ is velocity,
and $p$ pressure.
We assume an adiabatic pressure law
\begin{equation}\label{gaslaw}
p(v) = a_0 v^{-\gamma}
\end{equation}
corresponding to a $\gamma$-law gas,
for some constants $a_0>0$ and $\gamma \geq 1$.  
In the thermodynamical rarified gas approximation, $\gamma>1$ is the average over constituent particles of $\gamma=(n+2)/n$, where $n$ is the number of internal degrees of freedom of an individual particle \cite{Ba}: $n=3$ ($\gamma=1.66...$) for monatomic, $n=5$ ($\gamma=1.4$) for diatomic gas.  
For dense fluids, $\gamma$ is typically determined phenomenologically \cite{H}.
In general, $\gamma$ is usually taken within $1 \leq \gamma \leq 3$ 
in models of gas- or fluid-dynamical flow, 
whether phenomenological or derived by statistical 
mechanics \cite{Sm,Se.1,Se.2}.

%

It is well known that these equations support 
{\it viscous shock waves}, or asymptotically-constant traveling-wave solutions
\begin{equation}
\label{profile}
(v,u)(x,t)=(\bV, \bU)(x-s t),
\qquad
\lim_{z\to \pm \infty}(\bV,\bU)(z)=(v,u)_\pm,
\end{equation}
in agreement with physically-observed phenomena.
In nature, such waves are seen to be quite stable,
even for large variations in pressure between $v_\pm$.
However, it is a long-standing mathematical question 
to what extent this is reflected in the continuum-mechanical
model \eqref{psystem},
that is, for which choice of parameters
$(v_\pm,u_\pm,\gamma)$ are solutions of \eqref{profile} time-evolutionarily
stable in the sense of PDE; see, for example, the discussions in
\cite{IO,BE}.

The first result on this problem was obtained
by Matsumura and Nishihara in 1985 \cite{MN} using clever energy
estimates on the integrals of perturbations in $v$ and $u$,
by which they established stability with respect to the
restricted class of perturbations with ``zero mass'',
i.e., perturbations whose integral is zero, for shocks
with sufficiently small amplitude
$$
|p(v_+)-p(v_-)|\le C(v_-,\gamma),
$$
with $C\to \infty$ as $\gamma \to 1$, but $C<<\infty$ for $\gamma\ne 1$. 
There followed a number of works by Liu, Goodman, Szepessy-Xin,
and others \cite{L.1,Go.2,SX,L.2} toward the treatment
of general, nonzero mass perturbations;
see \cite{ZH,Z.2,Z.3} and references therein.
A complete result of stability with respect to general 
$L^1\cap H^3$ perturbations of small-amplitude shocks of 
system \eqref{psystem} was finally obtained in 2004 by 
Mascia and Zumbrun \cite{MZ.2} using pointwise semigroup 
techniques introduced by Zumbrun and Howard \cite{ZH,Z.4}
in the strictly parabolic case.


The result of \cite{MZ.2}, together with the small-amplitude spectral stability result 
of Humpherys and Zumbrun \cite{HuZ.2}\footnote{
The result of \cite{HuZ.1} is obtained by energy estimates
combining the techniques of \cite{MN} with those 
of \cite{Go.1,Go.2};
a similar approach has been used in \cite{LiuYu}
to obtain small-amplitude zero-mass stability of Boltzmann shocks.
See \cite{PZ,FS} for an alternative approach based on asymptotic
ODE methods.
}
generalizing that of \cite{MN},
in fact yields stability of small-amplitude 
shocks of general symmetric hyperbolic-parabolic systems, 
largely settling the problem of small-amplitude shock stability
for continuum mechanical systems.

However, there remains the interesting question of {\it large-amplitude
stability}.
The main result in this direction, following a general strategy
proposed in \cite{ZH},
is a ``refined Lyapunov theorem''\footnote{
``Refined'' because the linearized operator $L$ does not
possess a spectral gap, hence $e^{Lt}$ decays time-algebraically
and not exponentially; see \cite{ZH,Z.3} for further 
discussion.}
established by Mascia and Zumbrun \cite{MZ.4,Z.3} for general 
symmetric hyperbolic--parabolic systems, 
stating that linearized and nonlinear $L^1\cap H^3\to L^1\cap H^3$
orbital stability (the standard notions of stability) are equivalent
to {\it spectral stability},
or nonexistence of nonstable (nonnegative real part)
eigenvalues of the linearized operator $L$ about the wave,
other than the single zero eigenvalue arising through translational
invariance of the underlying equations.

%
This reduces the problem of large-amplitude stability to the
study of the associated eigenvalue equation $(L-\lambda)u=0$,
a standard analytically and numerically well-posed
(boundary value) problem in ODE, which can be attacked
by the large body of techniques developed for asymptotic,
exact, and numerical study of ODE.
In particular, there exist well-developed and efficient numerical algorithms
to determine the number of unstable roots
for any specific linearized operator $L$, independent of its
origins in the PDE setting;
see, e.g., \cite{Br.1,Br.2,BrZ,BDG,HuZ.2} and references therein.
In this sense, the problem of determining stability of any single
wave is satisfactorily resolved, or, for that matter, of any
compact family of waves.
To determine stability of a family of waves across an unbounded
parameter regime, however, is another matter.
It is this issue that we confront in attempting to decide 
the stability of general isentropic Navier--Stokes shocks.

As pointed out in \cite{ZH,HuZ.1}, 
zero-mass stability implies (and in a practical sense is
roughly equivalent to) spectral stability.  Thus, the
original results of Matsumura and Nishida
\cite{MN} imply small-amplitude shock
stability for general $\gamma$ and large-amplitude stability 
as $\gamma\to 1$.
Recently, Barker, Humpherys, Rudd, and Zumbrun \cite{BHRZ}
have carried out a numerical Evans function study
indicating stability on the large, but still bounded, parameter
range $\gamma \in [1,3]$, $1\le M\le 3,000$, where $M$
is the Mach number associated with the shock.
For discussion of the Evans function, see Section \ref{evanssec}. 
Recall that Mach number is an alternative measure of shock
strength, with $1$ corresponding to $|p(v_+)-p(v_-)|=0$ and
$M\to \infty$ corresponding to $|p(v_+)-p(v_-)|\to \infty$;
see Appendix A, \cite{BHRZ}, or Section \ref{profsec} below.
Mach $3,000$ is far beyond the hypersonic regime $M\sim 10^1$
encountered in current aerodynamics.
However, the mathematical question of stability across
arbitrary $\gamma$, $M$ remains open up to now.

In this paper, we resolve this question, using 
a combination of asymptotic ODE estimates and 
Evans function calculations to 
conclude,
first, stability
of isentropic Navier--Stokes shocks in the large Mach 
number limit $M\to \infty$ for any $\gamma\ge 1$,
and, second, stability for all $M\ge 2,500$ 
for $\gamma\in [1,2.5]$
(for $\gamma\in [1,2]$,
we obtain in fact stability for $M\ge 500$.)
The first result is obtained analytically, the second
by a systematic numerical study.
%
Together with the numerical results of \cite{BHRZ}, this
gives
convincing numerical evidence of
{\it unconditional stability for $\gamma \in [1,2.5]$,
independent of shock amplitude}.
As in \cite{BHRZ}, our numerical study {is not
a numerical proof}, but 
contains the necessary ingredients for one;
see discussion, Section \ref{quant}.
%
The restriction to $\gamma \in [1,2.5]$ is an arbitrary one
coming from the choice of parameters on which the numerical
study \cite{BHRZ} was carried out; stability for other $\gamma$
can be easily checked as well.

Our method of analysis is straightforward, though somewhat
delicate to carry out.
Working with the rescaled and conveniently rearranged
versions of the equations introduced in \cite{BHRZ},
we observe that the associated eigenvalue equations
converge uniformly as Mach number goes to infinity
on a ``regular region'' $x\le L$, for any fixed $L>0$,
to a limiting system that is well-behaved (hence treatable by
the standard methods of \cite{MZ.3,PZ,BHRZ}) in the sense
that its coefficient matrix converges uniformly exponentially
in $x$ to limits at $x=\pm \infty$, but is underdetermined
at $x=+\infty$.

On the complementary ``singular region'' $x\ge L$, the
convergence is only pointwise due to a fast ``inner structure''
featuring rapid variation of
the converging coefficient matrices near $x=+\infty$,
but the behavior at $x=+\infty$ is of course determinate.
Performing a boundary-layer analysis on the singular region
and matching across $x=L$, we are able to show convergence of the Evans function
of the original system as the Mach number goes to infinity
to an Evans function of the limiting system with 
an appropriately imposed additional condition at $x=+\infty$,
upstream of the shock.
This reduces the question of stability in the high-Mach number limit
to existence or nonexistence of zeroes of the limiting Evans function
on $\R \lambda \ge 0$,
a question that can be resolved by routine numerical computation
as in \cite{BHRZ},
or by energy estimates as in Appendix \ref{nonvanishing}.

The limiting system can be recognized as the eigenvalue equation
associated with a pressure-less ($\gamma=0$) gas, that is, 
the ``infinitely-compressible'' limit one might expect as the 
Mach number goes to infinity.  Recall that behavior in the low Mach number
limit is incompressible \cite{KM,KLN,Ho}.
However, the upstream boundary condition 
has to our knowledge
no such simple interpretation.
Indeed, to carry out the boundary-layer analysis 
by which we derive this condition 
is the main technical difficulty of the paper.

Besides their independent interest, the results of this paper  
seem significant as prototypes for future analyses.
Our calculations use some properties specific to the structure
of \eqref{psystem}.  In particular, we make use of surprisingly
strong energy estimates carried out in \cite{BHRZ} in
confining unstable eigenvalues 
to a bounded set independent of shock strength (or Mach number).
Also, we use the extremely simple structure of the
eigenvalue equation to carry out the key analysis of the eigenvalue
flow in the singular region near $x=+\infty$ essentially by hand.
However, these appear to be conveniences rather than essential
aspects of the analysis.
It is our hope that the basic argument structure of this paper
together with \cite{BHRZ}
can serve as a blueprint for the study of large-amplitude stability
in more general situations.

In particular, we expect that the analysis will carry over
to the full (nonisentropic) equations of gas dynamics 
with ideal gas equation of state, which, formally, 
decouple in the high-Mach number limit 
into the equations of isentropic pressureless gas dynamics
studied here,
augmented with a separate temperature equation governed by 
simple diffusion/Fourier's law.

\section{Preliminaries}\label{prelim}
We begin by recalling a number of preliminary steps carried out
in \cite{BHRZ}.
Making the standard change of coordinates $x \rightarrow x-s t$, 
we consider instead stationary solutions $(v,u)(x,t)\equiv (\bV,\bU)(x)$
of
\begin{equation}
\begin{split}
\label{movingpsystem}
v_t - s v_x - u_x &= 0,\\
u_t - s u_x + (a_0 v^{-\gamma})_x &= \left(\frac{u_x}{v}\right)_x.
\end{split}
\end{equation}
Under the rescaling 
\begin{equation}\label{scaling}
(x,t,v,u) \rightarrow 
(-\varepsilon sx, \varepsilon s^2 t, v/\varepsilon, -u/(\varepsilon s)),
\end{equation}
where $\varepsilon$ is chosen so that $0 < v_+ < v_- = 1$,  
our system takes the convenient form
\begin{equation}
\begin{split}
\label{rescaled}
v_t + v_x - u_x &= 0,\\
u_t + u_x + (a v^{-\gamma})_x &= \left(\frac{u_x}{v}\right)_x,
\end{split}
\end{equation}
where $a = a_0 \varepsilon^{-\gamma-1} s^{-2}$.  

\subsection{Profile equation}\label{profsec}
Steady shock profiles of \eqref{rescaled} satisfy 
\begin{equation*}
\begin{split}
v' - u' &= 0,\\
u' + (a v^{-\gamma})' &= \left(\frac{u'}{v}\right)',
\end{split}
\end{equation*}
subject to boundary conditions $(v,u)(\pm \infty) =(v_{\pm},u_{\pm})$, or, simplifying,
\begin{equation*}
v' + (a v^{-\gamma})' = \left(\frac{v'}{v}\right)'.
\end{equation*}

Integrating from $-\infty$ to $x$, we get the profile equation
\begin{equation}
\label{profeq}
v' =H(v, v_+):= v(v-1 + a  (v^{-\gamma}-1)),
\end{equation}
where $a$ is found by setting $x=+\infty$, thus yielding the Rankine-Hugoniot condition
\begin{equation}
\label{RH}
a = -\frac{v_+ - 1}{v_+^{-\gamma} - 1} = v_+^\gamma \frac{1-v_+}{1-v_+^\gamma}.
\end{equation}
Evidently, $a\to \gamma^{-1}$
in the weak shock limit $v_+\to 1$, while 
$ a\sim v_+^\gamma $ in the strong shock limit $v_+\to 0$.
The associated Mach number $M$ may be computed as in \cite{BHRZ}, 
Appendix A, as
\begin{equation}\label{mach}
M = (\gamma a)^{-1/2} 
\end{equation}
so that $M \sim \gamma^{-1/2} v_+^{-\gamma/2} \to +\infty$ as $v_+\to 0$
and $M\to 1$ as $v_+\to 1$; {\it that is, the high-Mach number limit
corresponds to the limit $v_+\to 0$.}

\subsection{Eigenvalue equations}\label{eigsec}
Linearizing \eqref{rescaled} about the profile $(\bV,\bU)$, we obtain
the eigenvalue problem
\begin{equation}
\label{eigen1}
\begin{split}
&\lambda v + v' - u' =0,\\
&\lambda u + u' - \left(\frac{h(\bV)}{\bV^{\gamma+1}}v\right)' = \left(\frac{u'}{\bV}\right)',
\end{split}
\end{equation}
where
\begin{equation}
\label{f}
h(\bV) = -\bV^{\gamma+1} + a(\gamma-1) + (a+1) \bV^\gamma.
\end{equation}

We seek nonstable eigenvalues
$\lambda \in \{\R(\lambda) \geq 0\}\setminus\{0\}$, i.e., $\lambda$ for which
\eqref{eigen1} possess a nontrivial
solution $(v,u)$ decaying at plus and minus spatial infinity.
As pointed out in \cite{ZH,HuZ.1}, by divergence form of the equations, 
such solutions necessarily satisfy $\int_{-\infty}^{+\infty}v(x)dx=
\int_{-\infty}^{+\infty}u(x)dx=0$, from which we may deduce that
\[
\tilde{u}(x) = \int_{-\infty}^x u(z) dz , \quad
\tilde{v}(x) = \int_{-\infty}^x v(z) dz
\]
and their derivatives decay exponentially as $x \rightarrow \infty$. 
Substituting and then integrating, we find that
$(\tilde{u},\tilde{v})$ satisfies the {\it integrated
eigenvalue equations} (suppressing the tilde)
\begin{subequations}\label{ep}
\begin{align}
&\lambda v + v' - u' =0, \label{ep:1}\\
&\lambda u + u' -  \frac{h(\bV)}{\bV^{\gamma+1}} v' = \frac{u''}{\bV}.\label{ep:2}
\end{align}
\end{subequations}
This new eigenvalue problem differs spectrally from \eqref{eigen1} only at $\lambda=0$, hence spectral stability of \eqref{eigen1} is equivalent to spectral stability of \eqref{ep}.  Moreover, since \eqref{ep} has no eigenvalue at $\lambda=0$, one can expect more uniform stability estimates for the integrated
equations in the vicinity of $\lambda=0$ \cite{MN,Go.1,ZH}.

\subsection{Preliminary estimates}\label{prelimests}
The following estimates established in \cite{BHRZ} 
indicate the suitability of the rescaling chosen in Section \ref{profsec}.
For completeness, we prove these in Appendix \ref{basicproof}.

\begin{proposition} [\cite{BHRZ}] \label{profdecay}
For each $\gamma\ge 1$, $0<v_+\le 1$, \eqref{profeq} has a unique 
(up to translation) monotone
decreasing solution $\hat v$ decaying to its endstates
with a uniform exponential rate, independent of $v_+$, $\gamma$.
In particular, for $0<v_+\le \frac{1}{12}$ and $\hat v(0):=v_+ + \frac{1}{12}$, 
\begin{subequations}
\label{decaybd}
\begin{align}
|\bV(x)-v_+|&\le \Big(\frac{1}{12}\Big)e^{-\frac{3x} {4}} \quad x\ge 0,\label{decaybd_1}\\
|\bV(x)-v_-|&\le 
\Big(\frac{1}{4}\Big)
e^{\frac{x+12}{2}} \quad x\le 0\label{decaybd_2}.
\end{align}
\end{subequations}
\end{proposition}

\begin{proposition}[\cite{BHRZ}] \label{hf}
Nonstable eigenvalues $\lambda$ of \eqref{ep}, i.e., eigenvalues
with nonnegative real part, are confined for any $0<v_+\le 1$ 
to the region
\begin{equation}
\label{hfbounds}
\R(\lambda) + |\I(\lambda)| \leq \Big(\sqrt{\gamma}+\frac{1}{2}\Big)^2.
\end{equation}
\end{proposition}

\subsection{Evans function formulation}\label{evanssec}
Following \cite{BHRZ}, we may express \eqref{ep} concisely
as a first-order system 
\begin{equation}\label{firstorder}
W' = A(x,\lambda) W,
\end{equation}
\begin{equation}
\label{evans_ode}
A(x,\lambda) = \begin{pmatrix}0 & \lambda & 1\\0 & 0 & 1\\ \lambda \bV& \lambda\bV &f(\bV)-\lambda \end{pmatrix},\quad W = \begin{pmatrix} u\\v\\v'\end{pmatrix},\quad \prime = \frac{d}{dx},
\end{equation}
where 
\begin{equation}\label{feq}
f(\bV) = \bV- \bV^{-\gamma} h(\bV)
= 2\bV - a(\gamma-1)\bV^{-\gamma} - (a+1),
\end{equation}
with $h$ as in \eqref{f} and $a$ as in \eqref{RH}, or, equivalently,
\begin{equation}\label{feq2}
f(\bV) = 2\bV - (\gamma-1) 
\Big(\frac{1-v_+}{1-v_+^\gamma}\Big) \Big(\frac{v_+}{\bV}\Big)^{\gamma} 
- \Big(\frac{1-v_+}{1-v_+^\gamma}\Big) v_+^{\gamma} -1.
\end{equation}

Eigenvalues of \eqref{ep} correspond to nontrivial solutions $W$ 
for which the boundary conditions $W(\pm\infty)=0$ are satisfied.  
Because $A(x,\lambda)$ as a function of $\bV$ is asymptotically constant 
in $x$, the behavior near $x=\pm \infty$ of solutions of 
\eqref{evans_ode} is governed by the limiting constant-coefficient
systems
\begin{equation}
\label{apm}
W' = A_\pm(\lambda) W, \qquad
A_\pm(\lambda):=A(\pm \infty,\lambda),
\end{equation}
from which we readily find on the (nonstable) domain $\R \lambda \ge 0$,
$\lambda\ne 0$ of
interest that there is a one-dimensional 
unstable manifold $W_1^-(x)$ of solutions decaying at $x=-\infty$ and 
a two-dimensional stable manifold $W_2^+(x) \wedge W_3^+(x)$ of
solutions decaying at $x=+\infty$, each of which may be chosen 
analytically in $\lambda$.  
With additional care, these may be extended analytically
to the whole set $\R \lambda\ge 0$, i.e., to $\lambda=0$ \cite{GZ}.
Defining the {\it Evans function} $D$ associated with operator $L$ as
the analytic function
\begin{equation}\label{evansdef}
D(\lambda): = \det(W_1^- W_2^+ W_3^+)_{\mid x=0},
\end{equation}
we find that eigenvalues of $L$ correspond to zeroes of $D$ both
in location and multiplicity; see, e.g., \cite{AGJ,GZ,MZ.3,Z.3} 
for further details.

Equivalently,
%
following \cite{PW,BHRZ}, we may express the Evans function as
\begin{equation}\label{adjevans}
D(\lambda)=\big(\widetilde{W}_1^+ \cdot W_1^-\big)_{\mid x=0},
\end{equation}
where $\widetilde{W}_1^+(x)$ spans the one-dimensional 
unstable manifold of solutions decaying at $x=+\infty$
(necessarily orthogonal to the span of 
$W_2^+(x)$ and  $W_3^+(x)$)
of the adjoint eigenvalue ODE
\begin{equation}\label{adjode}
\widetilde{W}' = 
- A(x,\lambda)^*  \widetilde{W}.
\end{equation}
The simpler representation \eqref{adjevans} 
is the one that we shall use here.

\section{Description of the main results}\label{description}

We can now state precisely our main results.

\subsection{Limiting equations}\label{limeqs}
Under the strategic rescaling \eqref{scaling}, 
both profile and eigenvalues equations converge {\it pointwise} as $v_+\to 0$
to limiting equations at $v_+=0$.
The limiting profile equation (the limit of \eqref{profeq}) is evidently
\begin{equation}
\label{limprofeq}
v' =v(v-1),
\end{equation}
with explicit solution
\begin{equation}\label{limprof}
\hat v_0(x)= \frac{1-\tanh (x/2)}{2},
\end{equation}
while the limiting eigenvalue system (the limit of \eqref{evans_ode}) is
\begin{equation}\label{limfirstorder}
W' = A^0(x,\lambda) W,
\end{equation}
\begin{equation}
\label{limevans_ode}
A^0(x,\lambda) = \begin{pmatrix}0 & \lambda & 1\\0 & 0 & 1\\ \lambda \bV_0& \lambda\bV_0 &f_0(\bV_0)-\lambda \end{pmatrix},
\end{equation}
where 
\begin{equation}\label{limfeq}
f_0(\bV_0) = 2\bV_0 - 1= -\tanh(x/2).
\end{equation}

Indeed, this convergence is {\it uniform} on any
interval $\hat v_0\ge \epsilon >0$, or, equivalently,
$x\le L$, for $L$ any positive constant, where
the sequence is therefore a {\it regular perturbation} of its limit.
We will call $x\in (-\infty,L]$ the ``regular region''
or ``regular side''.
For $\hat v_0\to 0$ on the other hand, or $x\to \infty$, the limit
is less well-behaved, as may be seen by the fact that $\partial f/\partial \hat v\sim \hat v^{-1}$ as $\hat v\to v_+$, a consequence of the
appearance of $\big(\frac{v_+}{\hat v}\big)$ in 
the expression \eqref{feq2} for $f$. 
Similarly, in contrast to $\hat v$, $A(x,\lambda)$ does not converge
to $A_+(\lambda)$ as $x\to +\infty$ with uniform exponential rate
independent of $v_+$, $\gamma$, but rather as $C\hat v^{-1}e^{-x/2}$.
We call $x\in [L,+\infty)$ therefore the ``singular region '' or
``singular side''.
(This is not a singular perturbation in the usual sense 
but a weaker type of singularity, at least as we have framed the
problem here.)

\subsection{Limiting Evans function}\label{redEvans}
We should now like to define a limiting Evans function following
the asymptotic Evans function framework introduced in \cite{PZ}
and establish convergence to this function in the $v_+\to 0$ limit,
thus reducing the stability problem as $v_+\to 0$ to the study of 
the (fixed) limiting Evans function.
However, we face certain difficulties due to the (mild) singularity of
the limit, as can be seen even at the first step of defining an Evans
function for the limiting system.

For, the limiting coefficient matrix 
\begin{equation}
\label{lim+}
A^0_+(\lambda) :=A^0(+\infty, \lambda)=
\begin{pmatrix}0 & \lambda & 1\\0 & 0 & 1\\ 0& 0 & -1-\lambda \end{pmatrix}
\end{equation}
is nonhyperbolic (in ODE sense) for all $\lambda$, 
having eigenvalues $0,0,-1-\lambda$; in particular,
the stable manifold drops to dimension one in the limit $v_+\to 0$.
Thus, the subspace in which $W_2^+$ and $W_3^+$ should be initialized
at $x=+\infty$ is not self-determined by \eqref{lim+}, 
but must be deduced by a careful 
study of the double limit $v_+\to 0$, $x\to +\infty$.
But, the computation
%
%
\begin{equation}
\label{Alim+}
\lim_{v_+\to 0}A(+\infty, \lambda)=
\begin{pmatrix}0 & \lambda & 1\\0 & 0 & 1\\ 0& 0 & -\gamma -\lambda \end{pmatrix}
\ne A^0_+(\lambda)=\lim_{x\to \infty}\lim_{v_+\to 0}A(x,\lambda)
\end{equation}
shows that these limits do not commute, except in the special case $\gamma=1$
already treated in \cite{MN} by other methods.

The rigorous treatment of this issue is the main work of the paper.
However, the end result can be easily motivated on heuristic grounds.
A study of $\lim_{v_+\to 0}A(+\infty, \lambda)$ on the
set $\R \lambda\ge 0$ of interest
reveals that eigenmodes decouple into a single ``fast'' (stable subspace)
decaying mode $(*,*,1)^T$ associated with eigenvalue $-\gamma-\lambda$
of strictly negative real part and a two-dimensional (center) subspace
of neutral modes $(r,0)^{T}$ associated with Jordan block
$
\begin{pmatrix} 0 & \lambda\\ 0 & 0 \end{pmatrix},
$
of which there is only a single genuine eigenvector $(1,0,0)^T$.
For $v_+$ small, therefore, $A_+(\lambda)$ has also a single
fast, decaying, eigenmode with eigenvalue near $-\gamma-1$ 
and two slow eigenmodes with eigenvalues near zero, one 
decaying and the other growing (recall, Section \ref{evanssec},
that the stable subspace of $A_+$ has dimension two
for $\R \lambda\ge 0$, $\lambda\ne 0$
and the unstable subspace dimension one).

Focusing on the single slow decaying eigenvector of $A_+$,
and considering its limiting behavior as $v_+\to 0$,
we see immediately that it must converge in direction to $\pm(1,0,0)^T$.
For, the sequence of direction vectors, since continuously varying
and restricted to a compact set, has
a nonempty, connected set of accumulation points, and these 
must be eigenvectors of $\lim_{v_+\to 0}A_+$ with eigenvalues near zero.
Since $\pm(1,0,0)^T$ are the unique candidates, we obtain the result.
Indeed, both growing and decaying slow eigenmodes must converge to this 
common direction, making the limiting analysis trivial.

The same argument shows that $\pm(1,0,0)^T$ is the limiting
direction of the slow stable eigenmode of $ A^0(x,\lambda)$
as $x\to +\infty$, that is, in the alternate limit
$\lim_{x\to \infty}\lim_{v_+\to 0}A(x,\lambda)$.
That is, $V_2^+:=(1,0,0)^T$ is the common limit of the
slow decaying eigenmode in either of the two alternative limits 
$\lim_{v_+\to 0}A_+$ and $\lim_{v_+\to 0}A_+$; 
it thus seems a reasonable choice to use this limiting slow direction 
to define an Evans function for the limiting system \eqref{limevans_ode}.
On the other hand, the stable eigenmode
$$V_3:=(a^{-1}(\lambda/a+1), a^{-1}, 1)^T, 
$$ 
$a=-1-\lambda$, of $A^0_+$ is forced on us by the system itself,
independent of the limiting process.

Combining these two observations,
we require that solutions $W^{0+}_2$ and $W^{0+}_3$ of the limiting eigenvalue
system \eqref{limevans_ode} lie asympotically in directions
$V_2$ and $V_3$, respectively, thus determining a limiting, or
``reduced'' Evans function 
\begin{equation}\label{limD}
D^0(\lambda):=
\det(W_1^{0-} W_2^{0+} W_3^{0+})_{\mid x=0},
\end{equation}
or alternatively
\begin{equation}\label{duallimD}
D^0(\lambda)=\big(\widetilde{W}_1^{0+} \cdot W_1^{0-}\big)_{\mid x=0},
\end{equation}
with $\widetilde{W}_1^{0+}$ defined analogously as a solution of
the adjoint limiting system 
lying asymptotically at $x=+\infty$ in direction
\begin{equation}\label{tildeV}
\widetilde V_1:= 
(0, 1, \bar a^{-1} )^T=
(0, 1, (1 +\bar \lambda)^{-1})^T
\end{equation}
orthogonal to the span of $V_2$ and $V_3$,
where `` $\bar{ }$ '' denotes complex conjugate.
(The prescription of $W_1^{0-}$ in the regular region is straightforward:
it must lie on the one-dimensional unstable manifold of $A_-^0$ as in the
$v_+>0$ case.)

\subsection{Physical interpretation}\label{phys}
Alternatively, the limiting equations may be derived by
taking a formal limit as $v_+\to 0$ of the rescaled equations
\eqref{rescaled}, recalling that $a\sim v_+^\gamma$, to obtain
a {\it limiting evolution equation}
\begin{equation}
\begin{split}
\label{pressureless}
v_t + v_x - u_x &= 0,\\
u_t + u_x  &= \left(\frac{u_x}{v}\right)_x
\end{split}
\end{equation}
corresponding to a {\it pressure-less gas}, or $\gamma=0$,
then deriving profile and eigenvalue equations from \eqref{pressureless}
in the usual way.
This gives some additional insight on behavior, of which
we make important mathematical use in Appendix \ref{nonvanishing}.
Physically, it has the interpretation that, in the high-Mach number
limit $v_+\to 0$, effects of pressure are concentrated near
$x=+\infty$ on the infinite-density side, as encoded
in the special upstream boundary condition $(u,u',v,v')\to c(1,0,0,0)$
as $x\to +\infty$ for the integrated eigenvalue equation,
which may be seen to be equivalent to the conditions imposed
on $W_j^+$ in the previous subsection.

\subsection{Analytical results}\label{results}
Defining $D^0$ as in \eqref{limD}--\eqref{duallimD},
we have the following main theorems.

\begin{theorem}\label{mainthm}
For $\lambda$ in any compact subset 
of $\R \lambda \ge 0$,
$D(\lambda)$ converges uniformly to $D^0(\lambda)$
as $v_+\to 0$.
\end{theorem}

\begin{corollary}\label{redcor}
For any compact subset $\Lambda$
of $\R \lambda \ge 0$,
$D$ is nonvanishing on $\Lambda$ for $v_+$ sufficiently small
if $D^0$ is nonvanishing on $\Lambda$, and is nonvanishing on
the interior of $\Lambda$ only if $D^0$ is nonvanishing there.
\end{corollary}

\begin{proof}
Standard properties of uniform limit of analytic functions.
\end{proof}

\begin{corollary}\label{stabcor}
Isentropic Navier--Stokes shocks are stable in
the high-Mach number limit $v_+\to 0$ if 
$D^0$ is nonvanishing on the wedge
\begin{equation}\label{Lambda}
\Lambda: \R(\lambda) + |\I(\lambda)| \leq \Big(\sqrt{\gamma}+\frac{1}{2}\Big)^2,
\quad \R \lambda \ge 0
\end{equation}
and only if $D^0$ is nonvanishing on the interior of $\Lambda$.
\end{corollary}

\begin{proof}
Corollary \ref{redcor} together with Proposition \ref{hf}.
\end{proof}

\begin{remark}\label{uni}
Likewise, on any compact subset of $\R \lambda \ge 0$, 
$|D|$ is uniformly bounded from zero
for $v_+$ sufficiently small ($M$ sufficiently large)
if and only if $|D^0|$ is uniformly bounded from zero.
Thus, isentropic Navier--Stokes shocks are ``uniformly stable''
for sufficiently small $v_+$, in the sense that $|D|$ is bounded
from below independent of $v_+$,
if and only if $D^0$ is nonvanishing on $\Lambda$ as defined in \eqref{Lambda}.
\end{remark}

The following result completes our abstract stability analysis.
The proof, given in Appendix \ref{nonvanishing},
is by an energy estimate analogous to that of \cite{MN}.\footnote{
Stability for $\gamma=1$, proved in \cite{MN},
already implies nonvanishing of $D^0$ outside the 
imaginary interval $[-i\sqrt{3/2}, +i\sqrt{3/2}]$, by Corollary \ref{stabcor}.}

\begin{proposition}\label{redenergy}
The limiting Evans function $D^0$ is nonzero on $\R \lambda \ge 0$.
\end{proposition}

\begin{corollary}\label{maincor}
For any $\gamma\ge 1$, isentropic Navier--Stokes shocks 
are stable for Mach number $M$ sufficiently large
(equivalently, $v_+$ sufficiently small).
\end{corollary}


\subsection{Numerical computations}\label{numerical}
Unfortunately, the energy estimate used to establish
\ref{redenergy}, though mathematically elegant, 
yields only the stated, abstract result
and not quantitative estimates.
A simpler and more general approach, that does 
yield quantitative information, is to compute the reduced Evans 
function numerically.
We carry out this by-now routine numerical computation
using the methods of \cite{BHRZ}.  Specifically, we map a
semicircle $\partial(\{\R \lambda \ge 0\}\cap \{|\lambda|\le 10\})$
enclosing $\Lambda$ for $\gamma\in [1,3]$ by $D^0$ and
compute the winding number of its image about the origin
to determine the number of zeroes of $D^0$ within the semicircle,
and thus within $\Lambda$.
%
For details of the numerical algorithm, 
see \cite{BHRZ,BrZ,HuZ.1}.


The result is displayed in Figure \ref{zero}, clearly indicating
stability.
%
%
%
%
More precisely, the minimum of $|D|$ on the semicircle is 
found to be $\approx 0.2433$.  Together with Rouch\'es Theorem, 
this gives explicit bounds on the size of the Mach number for
which shocks are stable, as displayed in Table \ref{comparison3},
Section \ref{quant}:
specifically, $M\ge 50$ for $\gamma \in [1,2]$,
$M\ge 2,400$ for $\gamma \in [2,2.5]$,
and $M\ge 13,000$ for $\gamma \in [2.5,3]$,
all corresponding approximately to $v_+=10^{-3}$.
In Figure \ref{zero2}, we superimpose on the image of the semicircle
by $D^0$ its (again numerically computed) image by 
the full Evans function $D$, for a monatomic gas $\gamma=1.66...$ 
at successively higher Mach numbers 
$v_+= ${\tt 1e-1,1e-2,1e-3,1e-4,1e-5,1e-6},
graphically demonstrating the convergence of $D$ to $D^0$
as $v_+$ approaches zero.

Indeed, we can see a great deal more from Figure \ref{zero2}.
For, note that the displayed contours are, to the scale visible
by eye, ``monotone'' in $v_+$, or nested, one within the other
(they do not appear so at smaller scales).
Thus, lower-Mach number contours are essentially ``trapped''
within higher-Mach number contours, with the worst-case, outmost
contour corresponding to the limiting Evans function $D^0$.
From this observation, we may conclude with confidence stability
down to the smallest value $M\approx 5.5$ displayed in the figure, 
corresponding (see Table \ref{comparison2}) to $v_+=10^{-1}$.
That is, a great deal of topological information is encoded
in the analytic family of Evans functions indexed by $v_+$,
from which stability may be deduced almost by inspection.
Behavior for other $\gamma\in [0,3]$ is similar. 
See, for example, the case $\gamma=3$
displayed in Figure \ref{zero3}, which is virtually identical
to Figure \ref{zero2}.\footnote{In particular, 
Figure \ref{zero3} indicates stability down to $v_+=10^{-1}$,
or Mach number $\sim 20$, from which we may conclude
unconditional stability on the whole range $\gamma\in [1,3]$ of \cite{BHRZ}.}

%
%

Such topological information 
does not seem to be available from other methods of investigating
stability such as 
direct discretation of the linearized operator about the wave \cite{KL} 
or studies based on linearized time-evolution or power methods \cite{BB,BMFC}.
{This represents in our view
a significant advantage of the Evans function formulation}.

\begin{figure}[t]
\begin{center}
\includegraphics[width=12cm]{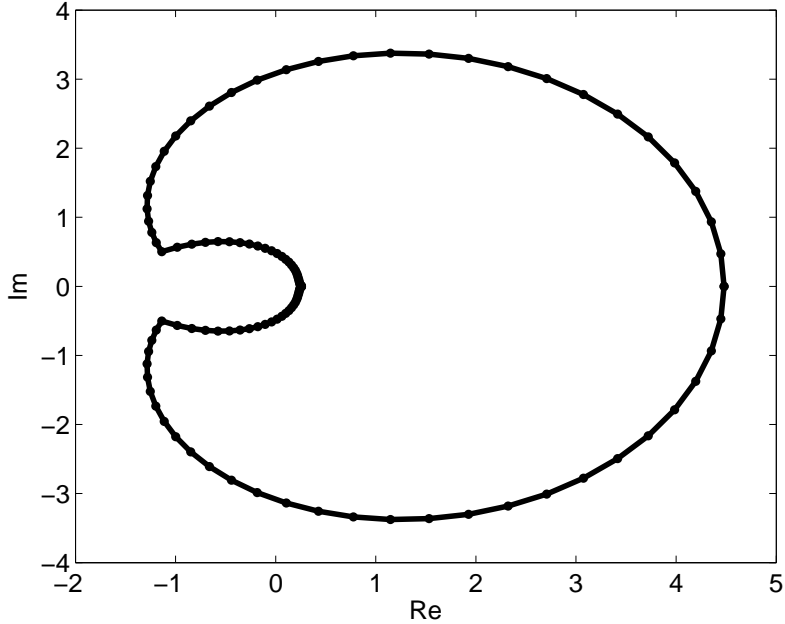} 
\end{center}
\caption{The image of the semi-circular contour via the Evans function for the reduced system.  
Note that the winding number of this graph is zero.  Hence, there are no unstable eigenvalues in the semi-circle.} 
\label{zero}
\end{figure}

\begin{figure}[t]
\begin{center}
\includegraphics[width=12cm]{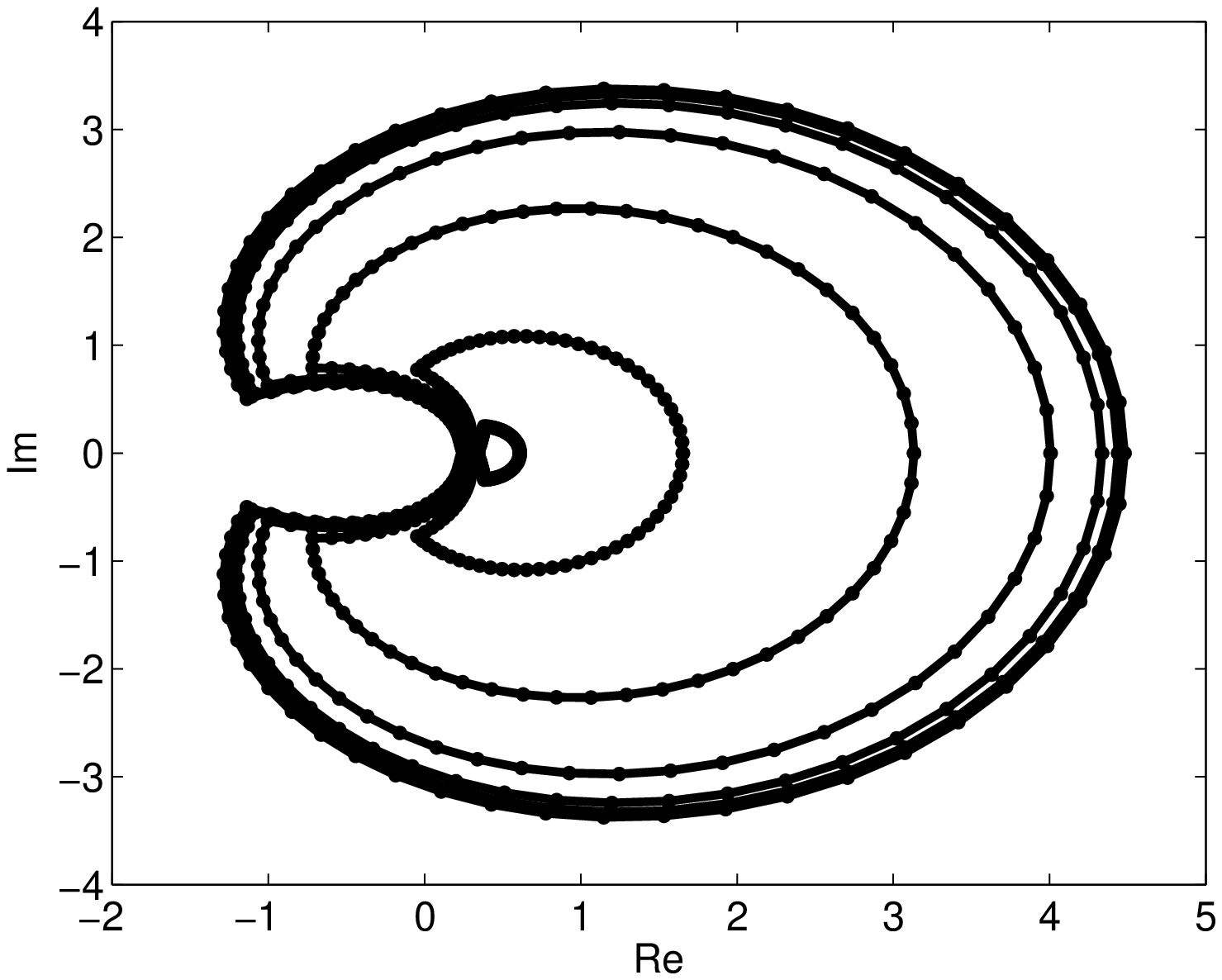} 
\end{center}
\caption{Convergence to the limiting Evans function 
as $v_+\to 0$ for a monatomic gas, $\gamma=1.66...$.
The contours depicted, going from inner to outer, 
are images of the semicircle under $D$
for $v_+=${\tt 1e-1,1e-2,1e-3,1e-4,1e-5,1e-6}.
The outermost contour is the image under $D^0$,
which is nearly indistinguishable from the image for $v_+=${\tt 1e-6}.}
\label{zero2}
\end{figure}

\begin{figure}[t]
\begin{center}
\includegraphics[width=12cm]{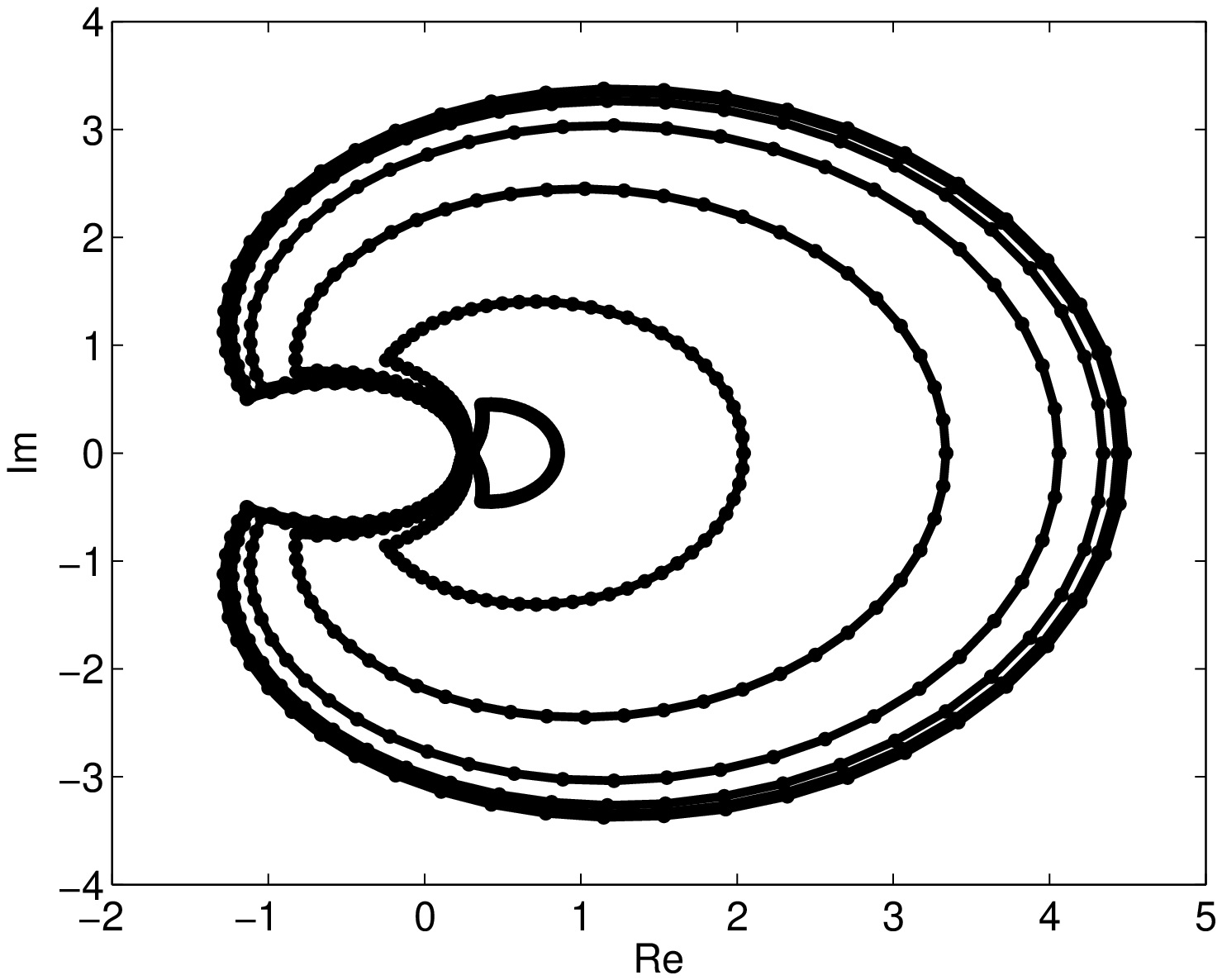} 
\end{center}
\caption{Convergence to the limiting Evans function 
as $v_+\to 0$ for $\gamma=1$.
The contours depicted, going from inner to outer, 
are images of the semicircle under $D$
for $v_+=${\tt 1e-1,1e-2,1e-3,1e-4,1e-5,1e-6}.
The outermost contour is the image under $D^0$,
which is nearly indistinguishable from the image for $v_+=${\tt 1e-6}.}
\label{zero1a}
\end{figure}

\begin{figure}[t]
\begin{center}
\includegraphics[width=12cm]{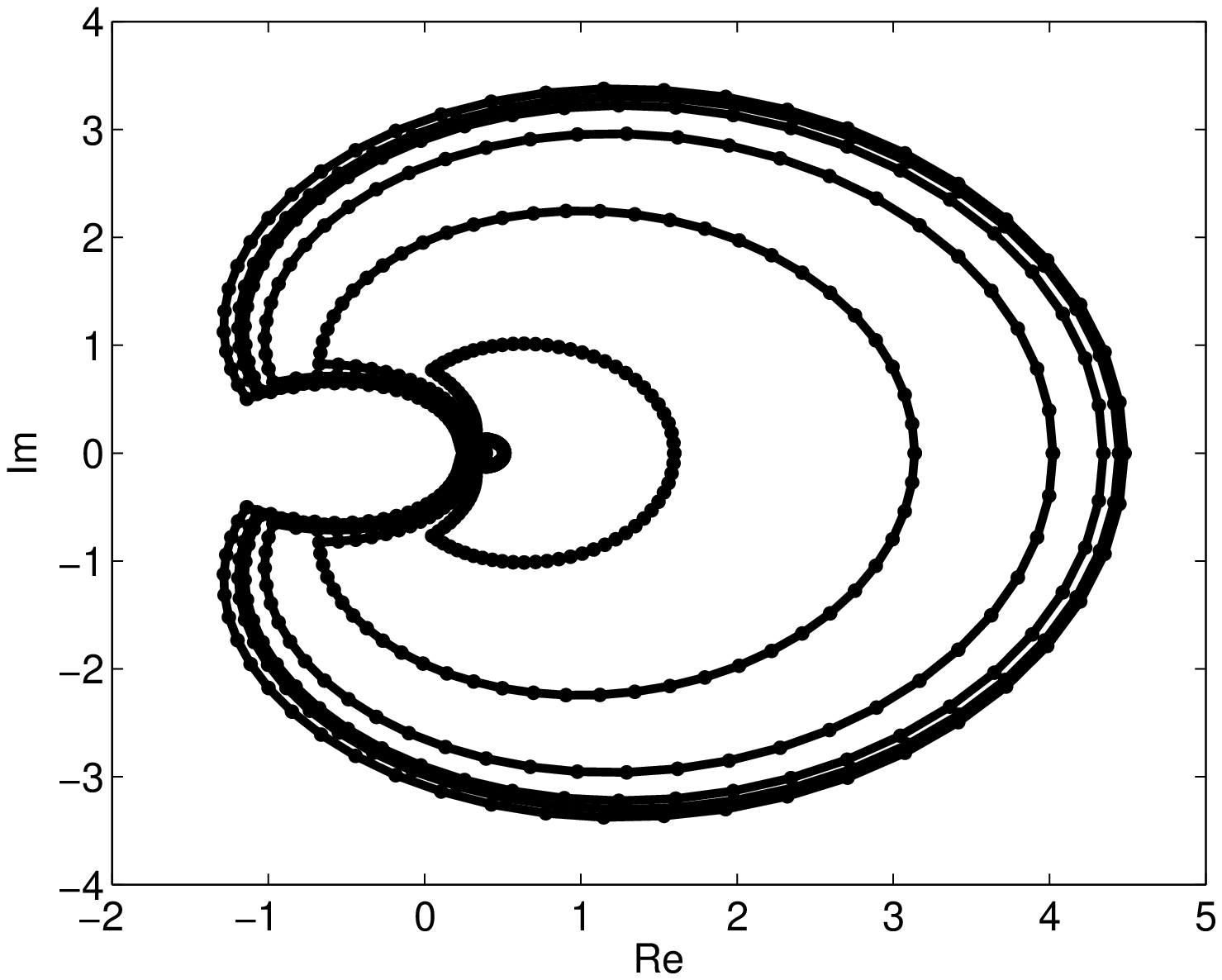} 
\end{center}
\caption{Convergence to the limiting Evans function 
as $v_+\to 0$ for $\gamma=3$.
The contours depicted, going from inner to outer, 
are images of the semicircle under $D$
for $v_+=${\tt 1e-1,1e-2,1e-3,1e-4,1e-5,1e-6}.
The outermost contour is the image under $D^0$,
which is nearly indistinguishable from the image for $v_+=${\tt 1e-6}.}
\label{zero3}
\end{figure}

\begin{remark}\label{goodchoice}
Recall that the Evans function is not determined
uniquely, but only up to a nonvanishing analytic factor \cite{AGJ,GZ}.
The simple contour structure in Figure \ref{zero2} is thus partly
due to a favorable choice of $D$ induced by the initialization at $\pm 
\infty$
by Kato's ODE \cite{Kato}, as described in \cite{BrZ,HuZ.2,BHRZ}.
A canonical algorithm for tracking bases of evolving subspaces, this in some
sense minimizes ``action''; see \cite{HSZ} for further discussion.
\end{remark}

\begin{remark}\label{ind}
Note that the limiting equations, and the limiting Evans function
$D^0$ are both independent of $\gamma$.
To study high-Mach number stability for $\gamma >3$, 
therefore, requires only to examine
$D^0$ on successively larger semicircles.
Thus, our methods in combination with the those of \cite{BHRZ}, 
allow us to determine stability in principle over any
bounded interval in $\gamma$, for $\gamma > 1$ and for all Mach numbers $M\ge 1$.
\end{remark}

\begin{remark}\label{altcalc}
As Figure \ref{zero2} suggests, an alternative method
for determining stability, without reference to $D^0$,
is to compute the full Evans function for sufficiently
high Mach number.  That is, nonvanishing of $D^0$, and thus
stability of sufficiently high Mach number shocks for $\gamma\in [1,3]$,
can already be concluded by large-but-finite Mach number
study of \cite{BHRZ} together with the fact that a limit
$D\to D^0$ exists (see also Remark \ref{uni}).
\end{remark}

\subsection{Conclusions}\label{conclusions}
The analytical result of Corollary \ref{maincor} guarantees stability 
for $\gamma\ge 1$, $M$ sufficiently large.
For $\gamma\in [1,2.5]$, our numerical results
indicate stability for $M\ge 2,500$ by a crude Rouch\/e bound,
and indeed much lower if further structure is taken into account.
Together with the small and intermediate Mach number studies of \cite{MN,BHRZ}
for $M\le 3,000$,
this yields unconditional stability of isentropic Navier--Stokes
shocks for $\gamma\in [1,2.5]$ and $M\ge 1$.
There is no inherent restriction to $\gamma\in [1,2.5]$;
as discussed in Remark \ref{ind}, numerical computations can be carried
out for any value of $\gamma$ to determine stability (or instability)
for all $M\ge 1$.
Indeed, our method of analysis indicates that the large-$\gamma$ limit
is quite analogous to the high-Mach number limit $v_+\to 0$, suggesting
the possibility to establish still more general results encompassing
all $\gamma\ge 1$, $M\ge 1$.

Our numerical results reveal also an unexpected ``universality''
of behavior in the high-Mach number regime, beyond just convergence
to the limiting system.  Namely, we see (cf. Figures \ref{zero2}
and \ref{zero3}) that behavior for a given $v_+$ is virtually
independent of the value of $\gamma$.  This also indicates 
that $v_+$ and not $M$ is the more useful measure of shock strength
in this regime.


\section{Boundary-layer analysis}\label{singular}

We now carry out the main work of the paper, analyzing
the flow of \eqref{evans_ode} in the singular region.
Our starting point is the observation that
\begin{equation}
\label{a+}
A(x,\lambda) = \begin{pmatrix}0 & \lambda & 1\\0 & 0 & 1\\ 
\lambda \hat v& \lambda \hat v &f(\hat v)-\lambda \end{pmatrix}
\end{equation}
is approximately block upper-triangular for $\hat v$ sufficiently small,
with diagonal blocks
$\begin{pmatrix}
0 & \lambda \\
0 &  0\\
\end{pmatrix}$
and
$\begin{pmatrix}
f(\hat v)-\lambda
\end{pmatrix}$
that are uniformly spectrally separated on $\R \lambda \ge 0$,
as follows by 
\begin{equation}\label{fneg}
f(\hat v)\le 2\hat v-1 \le -1/2.
\end{equation}
We exploit this structure by a judicious coordinate change
converting \eqref{evans_ode}
to a system in exact upper triangular form, for which the
decoupled ``slow'' upper lefthand $2\times 2$ block undergoes
a {\it regular perturbation} that can be analyzed by standard
tools introduced in \cite{PZ}.
Meanwhile, the fast, lower righthand $1\times 1$ block, since
scalar, may be solved exactly.

The global structure of this argument loosely follows the
general strategy of \cite{PZ} of first decoupling fast and slow
modes, then treating slow modes by regular perturbation methods.
However, there are interesting departures that may be of
use in other degenerate situations.  First, we only partially
decouple the system, to block-triangular rather than block-diagonal
form as in more standard cases, and second, we introduce a more
stable method of block-reduction taking account of usually
negligible derivative terms in the definition of block-triangularizing 
transformations, which, if ignored, would in this case lead to unacceptably 
large errors.

\subsection{Preliminary transformation}\label{pretrans}
We first block upper-triangularize by a static (constant) coordinate
transformation the limiting matrix
\begin{equation}
\label{lima+}
A_+=A(+\infty,\lambda) = \begin{pmatrix}0 & \lambda & 1\\0 & 0 & 1\\ 
\lambda v_+& \lambda v_+ &f(v_+)-\lambda \end{pmatrix}
\end{equation}
at $x=+\infty$ using special block lower-triangular transformations
\begin{equation}\label{statictrans}
R_+:=\begin{pmatrix}
I& 0\\
\lambda v_+ \theta_+ & 1\\
\end{pmatrix},
\qquad
L_+:=R_+^{-1}=\begin{pmatrix}
I& 0\\
-\lambda  v_+\theta_+ & 1\\
\end{pmatrix},
\end{equation}
where $I$ denotes the $2\times 2$ identity matrix and
$\theta_+\in \C^{1\times 2}$ is a $1\times 2$ row vector.

\begin{lemma}\label{pretranslem}
On any compact subset of $\R \lambda \ge 0$, for each $v_+>0$ 
sufficiently small, 
there exists a unique $\theta_+=\theta_+(v_+,\lambda)$ such that
$\hat A_+:=L_+A_+R_+$ is upper block-triangular, 
\begin{equation}\label{hata+}
\begin{aligned}
\hat A_+&= 
\begin{pmatrix}
\lambda (J+ v_+\bb1 \theta_+) &  \bb1 \\
0 &  f(\hat v)-\lambda -\lambda v_+ \theta_+\bb1 \\
\end{pmatrix},
\end{aligned}
\end{equation}
where $J=\begin{pmatrix} 0 & 1 \\ 0 & 0 \end{pmatrix}$
and 
$\bb1=\begin{pmatrix} 1 \\ 1 \\ \end{pmatrix} $,
satisfying a uniform bound
\begin{equation}\label{theta+bd}
|\theta_+|\le C.
\end{equation}
\end{lemma}

\begin{proof}
Setting the $2-1$ block of $\hat A_+$ to zero, we obtain the
matrix equation
$$
\theta_+ (aI-\lambda J)
=  -\bb1^T + \lambda v_+ \theta_+ \bb1 \theta_+,
$$
where $a=f(v_+)-\lambda$, or, equivalently, the fixed-point equation
\begin{equation}\label{fix1}
\theta_+ =
(aI-\lambda J)^{-1}
\Big( -\bb1^T + \lambda v_+\theta_+ \bb1 \theta_+\Big).
\end{equation}
By $\det (aI-\lambda J)= a^2\ne 0$,
$(aI-\lambda J)^{-1}$ is uniformly bounded
on compact subsets of $\R \lambda \ge 0$
(indeed, it is uniformly bounded on all of $\R \lambda \ge 0$),
whence, for $|\lambda|$ bounded and $v_+$ sufficiently small, 
there exists a unique
solution by the Contraction Mapping Theorem,
which, moreover, satisfies \eqref{theta+bd}.
\end{proof}

\subsection{Dynamic triangularization}\label{dynamic}
Defining now $Y:=L_+W$ and
%
%
\begin{equation}\label{hata}
\begin{aligned}
\hat A(x,\lambda)&= L_+A(x, \lambda) R_+(x,\lambda)\\
&=
\begin{pmatrix}\\
\lambda (J+v_+\bb1 \theta_+) &  \bb1 \\
\lambda (\hat v-v_+) \bb1^T 
  -\lambda v_+(f(\hat v)-f(v_+))\theta_+
&  f(\hat v)-\lambda -\lambda v_+ \theta_+\bb1 \\
\end{pmatrix},
\end{aligned}
\end{equation}
we have converted \eqref{evans_ode} to an asymptotically block
upper-triangular system
\begin{equation}\label{tri1}
Y'=\hat A(x,\lambda) Y, 
\end{equation}
with $\hat A_+=\hat A(+\infty, \lambda)$ as in \eqref{hata+}.
Our next step is to choose a {\it dynamic} transformation of
the same form
\begin{equation}\label{dyntrans}
\tilde R:=\begin{pmatrix}
I& 0\\
\tilde \Theta & 1\\
\end{pmatrix},
\qquad
\tilde L:=\tilde R^{-1}=\begin{pmatrix}
I& 0\\
-\tilde \Theta & 1\\
\end{pmatrix},
\end{equation}
converting \eqref{tri1} to an exactly block upper-triangular
system, with $\tilde \Theta$ uniformly exponentially decaying
at $x=+\infty$: that is, a {\it regular perturbation} of the identity.

\begin{lemma}\label{dyntranslem}
On any compact subset of $\R \lambda \ge 0$, 
for $L$ sufficiently large and each $v_+>0$ sufficiently small, 
there exists a unique $\Theta=\Theta_+(x,\lambda, v_+)$ such that
$\tilde A:=\tilde L \hat A(x,\lambda)\tilde R
+ \tilde L'\tilde R$ is upper block-triangular, 
\begin{equation}\label{ahat}
\begin{aligned}
\tilde A&= 
\begin{pmatrix}
\lambda (J+v_+\bb1 \theta_+) + \bb1 \tilde \Theta &  \bb1 \\
0 &  f(\hat v)-\lambda -\lambda \theta_+\bb1 -\tilde \Theta \bb1 \\
\end{pmatrix},
\end{aligned}
\end{equation}
and 
$\tilde \Theta(L)=0$, satisfying a uniform bound
\begin{equation}\label{Thetabd}
|\tilde \Theta(x,\lambda, v_+)|\le Ce^{-\eta x},
\qquad \eta>0, \, x\ge L,
\end{equation}
independent of the choice of $L$, $v_+$.
\end{lemma}

\begin{proof}
Setting the $2-1$ block of $\tilde A$ to zero and computing
$$
\tilde L'\tilde R=
\begin{pmatrix}
0 & 0\\
-\tilde \Theta' & 0
\end{pmatrix}
\begin{pmatrix}
I & 0\\
\tilde \Theta & I
\end{pmatrix}
=
\begin{pmatrix}
0 & 0\\
-\tilde \Theta' & 0,
\end{pmatrix}
$$
we obtain the matrix equation
\begin{equation}\label{mat}
\tilde \Theta' - \tilde \Theta \big(aI-\lambda (J +v_+\bb1\theta_+)\big)
=  \zeta+ \tilde \Theta \bb1 \tilde \Theta,
\end{equation}
where the forcing term
$$
\zeta:=
-\lambda (\hat v-v_+) \bb1^T 
  +\lambda v_+(f(\hat v)-f(v_+))\theta_+
$$
by derivative estimate $df/d\hat v\le C\hat v^{-1}$
together with the Mean Value Theorem 
is uniformly exponentially decaying:
\begin{equation}\label{zetabd}
\begin{aligned}
|\zeta|\le C |\hat v-v_+|\le
C_2 e^{-\eta x},
\qquad
\eta>0.
\end{aligned}
\end{equation}

Initializing $\tilde \Theta(L)=0$, we obtain by Duhamel's Principle/Variation
of Constants the representation (supressing the argument $\lambda$)
\begin{equation}\label{duhamel}
\tilde \Theta(x)=
\int_L^{x}
S^{y\to x}
(\zeta+ \tilde \Theta \bb1 \tilde \Theta)(y)
\, dy,
\end{equation}
where $S^{y\to x}$ is the solution operator for the homogeneous
equation
$$
\tilde \Theta' - \tilde \Theta \big(aI-\lambda (J +v_+\bb1\theta_+)\big)=0,
$$
or, explicitly,
$$
S^{y\to x}=
e^{\int_y^x a(y)dy}
e^{ -\lambda (J +v_+\bb1\theta_+)(x-y)}.
$$

For $|\lambda|$ bounded and $v_+$ sufficiently small, 
we have by matrix perturbation theory
that the eigenvalues of $ -\lambda (J +v_+\bb1\theta_+)$
are small and the entries are bounded, hence
$$ 
|e^{ -\lambda (J +v_+\bb1\theta_+)z}|\le Ce^{\epsilon z}
$$
for $z\ge 0$.  Recalling the uniform spectral gap
$\R a =f(\hat v)-\R \lambda \le -1/2$ for $\R \lambda \ge 0$,
we thus have
\begin{equation}\label{Sbd}
|S^{y\to x}|\le Ce^{\eta (x-y)}
\end{equation}
for some $C$, $\eta>0$.
Combining \eqref{zetabd} and \eqref{Sbd}, we obtain
\begin{equation}
\begin{aligned}
\Big|\int_L^{x} S^{y\to x} \zeta(y)\, dy\Big|&\le
\int_L^x
C_2 e^{-\eta(x-y)}e^{-(\eta/2) y} dy \\
& =C_3 e^{-(\eta/2)x}.
\end{aligned}
\end{equation}

Defining $\tilde \Theta(x) =:\tilde \theta(x) e^{-(\eta/2)x}$
and recalling \eqref{duhamel} we thus have
\begin{equation}\label{duhamel2}
\tilde \theta(x)=
f + e^{(\eta/2)x}\int_L^{x}
S^{y\to x}
 e^{-\eta y}\tilde \theta \bb1 \tilde \theta(y)
\, dy,
\end{equation}
where $f:= e^{(\eta/2)x}\int_L^{x} S^{y\to x} \zeta(y) \, dy$
is uniformly bounded, $|f|\le C_3$, and
$
e^{(\eta/2)x}\int_L^{x} S^{y\to x}
e^{-\eta y}\tilde \theta \bb1 \tilde \theta(y)
\, dy
$
is contractive with arbitrarily small contraction constant $\epsilon>0$ 
in $L^\infty[L,+\infty)$ for $|\tilde \theta|\le 2C_3$ for $L$ sufficiently
large, by the calculation
$$
\begin{aligned}
&\Big| e^{(\eta/2)x}\int_L^{x} S^{y\to x}
e^{-\eta y}\tilde \theta_1 \bb1 \tilde \theta_1(y)
-
e^{(\eta/2)x}\int_L^{x} S^{y\to x}
e^{-\eta y}\tilde \theta_2 \bb1 \tilde \theta_2(y)
\Big| 
\\
& \qquad \qquad \le
\Big|e^{(\eta/2)x}\int_L^{x} Ce^{-\eta(x-y)} e^{-\eta y} \, dy \Big| 
\|\tilde \theta_1- \tilde \theta_2\|_\infty  
\max_j \|\tilde \theta_j\|_\infty  
\\
& \qquad \qquad \le
e^{-(\eta/2)L}\Big|\int_L^{x} Ce^{-(\eta/2)(x-y)}  \, dy \Big| 
\|\tilde \theta_1- \tilde \theta_2\|_\infty  
\max_j \|\tilde \theta_j\|_\infty \\
& \qquad \qquad =
C_3e^{-(\eta/2)L} \|\tilde \theta_1- \tilde \theta_2\|_\infty  
\max_j \|\tilde \theta_j\|_\infty. 
\end{aligned}
$$
It follows by the Contraction Mapping Principle that there exists
a unique solution $\tilde \theta$ of fixed point equation
\eqref{duhamel2} with $|\tilde \theta(x)|\le 2C_3$
for $x\ge L$,
or, equivalently (redefining the unspecified constant $\eta$), \eqref{Thetabd}.
\end{proof}

\begin{remark}\label{delicate}
The above calculation is the most delicate part of the analysis,
and the main technical point of the paper.
The interested reader may verify that a ``quasi-static'' transformation
treating term $\tilde \Theta'$ in \eqref{mat} as an error, as is typically
used in situations of slowly-varying coefficients (see for example
\cite{MZ.3,PZ}), would lead to unnaceptable errors of magnitude
$$
O(|f(\hat v))'||\hat v-v_+|)=
O(|df/d\hat v)| |\hat v-v_+|)=
O(|\hat v|^{-1} |\hat v-v_+|).
$$
One may think of the exact ODE solution \eqref{dyntrans}
as ``averaging'' the effects
of rapidly-varying coefficients by integration of \eqref{mat}.
\end{remark}

\subsection{Fast/Slow dynamics}\label{slow}
Making now the further change of coordinates 
$$
Z=\tilde LY
$$
and computing
$$
\begin{aligned}
(\tilde LY)'=\tilde L Y' + \tilde L' Y
&=(\tilde LA_++\tilde L')Y,\\
&=(\tilde LA_+\tilde R+\tilde L'\tilde R)Z,\\
\end{aligned}
$$
we find that we have converted \eqref{tri1} to a block-triangular system
\begin{equation} \label{tri2}
Z'=\tilde AZ=
\begin{pmatrix}
\lambda (J+v_+\bb1 \theta_+) + \bb1 \tilde \Theta &  \bb1 \\
0 &  f(\hat v)-\lambda -\lambda v_+ \theta_+\bb1 -\tilde \Theta \bb1 \\
\end{pmatrix}Z,
\end{equation}
related to the original eigenvalue system \eqref{evans_ode} by
\begin{equation}\label{WZ}
W=LZ,\quad
R:=R_+R=
\begin{pmatrix}
I & 0\\
\Theta & 1
\end{pmatrix},
\quad 
L:=R^{-1}=
\begin{pmatrix}
I & 0\\
-\Theta & 1
\end{pmatrix},
\end{equation}
where
\begin{equation}\label{Theta}
\Theta= \tilde \Theta + \lambda v_+ \theta_+.
\end{equation}

Since it is triangular, \eqref{tri2} may be solved completely
if we can solve the component systems associated with its diagonal
blocks.  The {\it fast system}
$$
z'=
\Big(f(\hat v)-\lambda -\lambda v_+ \theta_+\bb1 -\tilde \Theta \bb1 \Big)z
$$
associated to the lower righthand block features rapidly-varying
coefficients.  However, because it is scalar, 
it can be solved explicitly by exponentiation.

The {\it slow system }
\begin{equation} \label{slowsys}
z'= \Big(\lambda (J+v_+\bb1 \theta_+) + \bb1 \tilde \Theta \Big) z
\end{equation}
associated to the upper lefthand block, on the other hand,
by \eqref{Thetabd}, is an exponentially decaying
perturbation of a constant-coefficient system
\begin{equation}\label{cc}
z'= \lambda (J+v_+\bb1 \theta_+)z
\end{equation}
that can be explicitly solved by exponentiation, and thus
can be well-estimated by comparison with \eqref{cc}.
A rigorous version of this statement is given by the 
{\it conjugation lemma} of \cite{MeZ}:

\begin{proposition}[\cite{MeZ}] \label{conjugation}
Let $M(x,\lambda)=M_+(\lambda)+ \Theta(x,\lambda)$,
with $M_+$ continuous in $\lambda$ and $|\Theta(x,\lambda)|\le Ce^{-\eta x}$,
for $\lambda$ in some compact set $\Lambda$.
Then, there exists a globally invertible matrix
$P(x,\lambda)=I + Q(x,\lambda)$ such that the 
coordinate change $z=Pv$ converts the variable-coefficient ODE 
$
z'=M(x,\lambda)z
$
to a constant-coefficient equation
$$
v'=M_+(\lambda)v,
$$
satisfying for any $L$, $0<\hat \eta < \eta$ a uniform bound
\begin{equation}\label{qdecay}
|Q(x,\lambda)|\le 
C(L,\hat \eta, \eta, \max |(M_+)_{ij}|, \dim M_+)e^{-\hat \eta x}
\quad \hbox{for $x\ge L$}.
\end{equation}
\end{proposition}

\begin{proof}
See \cite{MeZ,Z.3}, or Appendix \ref{quantlem}.
\end{proof}

By Proposition \ref{conjugation}, the solution operator for \eqref{slowsys}
is given by
\begin{equation}\label{slowsoln}
P(y,\lambda) e^{\lambda (J+v_+\bb1 \theta_+(\lambda, v_+))(x-y)}
P(x,\lambda)^{-1},
\end{equation}
where $P$ is a uniformly small perturbation of the identity
for $x\ge L$ and $L>0$ sufficiently large.

\section{Proof of the main theorem}\label{proofsec}

With these preparations, we turn now to the proof of the main theorem.

\subsection{Boundary estimate}\label{estW1}
We begin by establishing the following key estimates on
$\widetilde W_1^+(L)$, that is, the value of
the dual mode $\widetilde W_1^+$ appearing in \eqref{adjevans} 
at the boundary $x=L$ between regular and singular regions.

\begin{lemma}\label{matching}
For $\lambda$ on any compact subset of $\R \lambda \ge 0$,
and $L>0$ sufficiently large,
with $\widetilde W_1^+$ normalized as in \cite{GZ,PZ,BHRZ},
\begin{equation}\label{wcon}
|\widetilde W_1^+(L)-\widetilde V_1| \le Ce^{-\eta L}
\end{equation}
as $v_+\to 0$, uniformly in $\lambda$, where $C$, $\eta>0$ are
independent of $L$
and 
$$
\widetilde V_1:= (0, 1, (1 +\bar \lambda)^{-1})^T
$$
is the limiting direction vector \eqref{tildeV}
appearing in the definition of $D^0$.
\end{lemma}

\begin{corollary}\label{matching2}
Under the hypotheses of Lemma \ref{matching},
\begin{equation}\label{limcon}
|\tilde W_1^{0+}(L)-\widetilde V_1| \le Ce^{-\eta L}
\end{equation}
and
\begin{equation}\label{wcon2}
|\widetilde W^{+}_1(L) -\widetilde W^{0+}_1(L)|\le Ce^{-\eta L}
\end{equation}
as $v_+\to 0$, uniformly in $\lambda$, where $C$, $\eta>0$ are
independent of $L$ and $\widetilde W_1^{0+}$ is the solution
of the limiting adjoint eigenvalue system
appearing in definition \eqref{duallimD} of $D^0$.
\end{corollary}

\begin{proof}[Proof of Lemma \ref{matching}]
Making the coordinate-change
\begin{equation}\label{dualL}
\tilde Z:= R^* \tilde W,
\end{equation}
$R$ as in \eqref{WZ}, reduces the adjoint equation 
$\tilde W'=-A^*\tilde W$ to block lower-triangular form,
%
%
\begin{equation} \label{dualtri2}
\begin{aligned}
\tilde Z'&=-\tilde A^* \tilde Z\\
&=
\begin{pmatrix}
-\bar \lambda J^T +
(\lambda v_+\bb1 \theta_+ - \bb1 \tilde \Theta)^* &  0 \\
-\bb1^T &  -f(\hat v)+\bar \lambda +\bar \lambda v_+ (\theta_+\bb1)^* 
-(\tilde \Theta \bb1)^* \\
\end{pmatrix}Z,
\end{aligned}
\end{equation}
with `` $\bar{ }$ '' denoting complex conjugate.

Denoting by $\tilde V^+_1$ a suitably normalized 
element of the one-dimensional (slow) stable subspace
of $-\tilde A^*$, we find, similarly as in the discussion of Section
\ref{redEvans} that, without loss of generality, 
\begin{equation}\label{dualVlim}
\tilde V^+_1 \to (0, 1, (\gamma+\bar \lambda)^{-1})^T
\end{equation} 
as $v_+\to 0$, while the associated eigenvalue $\tilde \mu_1^+\to 0,$
uniformly for $\lambda$ on an compact subset of $\R \lambda\ge 0$.
The dual mode $\tilde Z_1^+=R^* \tilde W_1^+$ is uniquely determined
by the property that it is asymptotic as $x\to +\infty$
to the corresponding constant-coefficient solution 
$e^{\tilde \mu_1^+}\tilde V_1^+$ 
(the standard normalization of \cite{GZ,PZ,BHRZ}).

By lower block-triangular form \eqref{dualtri2}, the equations
for the slow variable $\tilde z^T:=(\tilde Z_1, \tilde Z_2)$ decouples
as a slow system
\begin{equation} \label{dualslowsys}
\tilde z'= -\Big(\lambda (J+v_+\bb1 \theta_+) + \bb1 \tilde \Theta \Big)^* 
\tilde z
\end{equation}
dual to \eqref{slowsys}, with solution operator
\begin{equation}\label{dualslowsoln}
P^*(x,\lambda)^{-1} e^{-\bar \lambda (J+v_+\bb1 \theta_+)^*)(x-y)}
P(y,\lambda)^{*}
\end{equation}
dual to \eqref{slowsoln}, i.e. (fixing $y=L$, say), solutions of general form
\begin{equation}\label{genform}
\tilde z(\lambda,x)=
P^*(x,\lambda)^{-1} e^{-\bar \lambda (J+v_+\bb1 \theta_+)^*)(x-y)}
\tilde v,
\end{equation}
$\tilde v \in \C^2$ arbitrary.

Denoting by
$$
\tilde Z_1^+(L):=R^*\tilde W_1^+(L),
$$
therefore, the unique (up to constant factor) decaying solution
at $+\infty$, and 
$\tilde v_1^+:=((\tilde V_1^+)_1 , (\tilde V_1^+)_2)^T$, 
we thus have evidently
$$
\tilde z_1^+(x,\lambda)=
P^*(x,\lambda)^{-1} e^{-\bar \lambda (J+v_+\bb1 \theta_+)^*)x}
\tilde v_1^+,
$$
which, as $v_+\to 0$, is uniformly bounded by
\begin{equation}\label{weakexp}
|\tilde z_1^+(x,\lambda)|\le Ce^{\epsilon x}
\end{equation}
for arbitrarily small $\epsilon>0$
and, by \eqref{dualVlim}, converges for $x$ less than or equal to
any fixed $X$ simply to 
\begin{equation}\label{simplelim}
\lim_{v_+\to 0}\tilde z_1^+(x,\lambda)=
P^*(x,\lambda)^{-1} (0,1)^T.
\end{equation}

Defining by $\tilde q:=(\tilde Z_1^+)_3$
the fast coordinate of $\tilde Z_1^+$, we have, by \eqref{dualtri2},
$$
\tilde q'
+\Big(f(\hat v)-\bar \lambda -(\lambda v_+ \theta_+\bb1 + \tilde \Theta \bb1)^* \Big)
\tilde q=
\bb1^T \tilde z_1^+,
$$
whence, by Duhamel's principle, any decaying solution is given by
$$
\tilde q(x,\lambda)=
\int_x^{+\infty} e^{\int_y^x a(z,\lambda, v_+)dz}\bb1^T z_1^+(y) \, dy,
$$
where
$$
a(y,\lambda,v_+):=
-\Big(f(\hat v)-\bar \lambda -(\lambda v_+ \theta_+\bb1 + \tilde \Theta \bb1)^* \Big).
$$
Recalling, for $\R \lambda \ge 0$, that $\R a \ge 1/2$, combining
\eqref{weakexp} and \eqref{simplelim}, 
and noting that $a$ converges uniformly on $y\le Y$ as $v_+\to 0$ for 
any $Y>0$ to 
$$
\begin{aligned}
a_0(y, \lambda)&:=
-f_0(\hat v)+\bar \lambda 
+(\tilde \Theta_0 \bb1)^* \\
&=  (1+\bar \lambda)
+O(e^{-\eta y})
\end{aligned}
$$
we obtain by the Lebesgue
Dominated Convergence Theorem that
$$
\begin{aligned}
\tilde q(L,\lambda)&\to  
\int_L^{+\infty} e^{\int_y^L a_0(z,\lambda)dz}\bb1^T (0,1)^T \, dy\\
&=
\int_L^{+\infty} 
e^{(1+\bar \lambda)(L-y)+ \int_y^L O(e^{-\eta z})dz }
\, dy\\
&=
(1+\bar \lambda)^{-1}(1 +O(e^{-\eta L})).
\end{aligned}
$$
Recalling, finally, \eqref{simplelim}, and the fact that
$$
|P-Id|(L,\lambda), \,  |R-Id|(L,\lambda) \le Ce^{-\eta L}
$$
for $v_+$ sufficiently small, we obtain \eqref{wcon} as claimed.
\end{proof}

\begin{proof}[Proof of Corollary \ref{matching2}]
Applying Proposition \ref{conjugation} to the limiting adjoint system 
$$
\tilde W'=-(A^0)^* \tilde W=
\begin{pmatrix}0 & 0 & 0\\-\bar \lambda & 0 & 0\\ 
-1& -1 & 1+\bar \lambda \end{pmatrix}
\tilde W + O(e^{-\eta x})\tilde W,
$$
we find that, up to an $Id +O(e^{-\eta x})$ coordinate change,
$\tilde W_1^{0+}(x)$ is given by the exact solution
$\tilde W\equiv \tilde V_1$ of the limiting, constant-coefficient
system
$$
\tilde W'=-(A^0)^* \tilde W=
\begin{pmatrix}0 & 0 & 0\\-\bar \lambda & 0 & 0\\ 
-1& -1 & 1+\bar \lambda \end{pmatrix}
\tilde W.
$$
This yields immediately \eqref{limcon},
which, together with \eqref{wcon}, yields \eqref{wcon2}.
\end{proof}

\begin{remark}\label{num}
Noting that \eqref{limcon} is sharp, we see 
from \eqref{wcon2} that the error between $\tilde W_1^+(L)$
and $\tilde W_1^{0+}(L)$ is already within the error tolerance 
of the numerical scheme used to approximate $D^0$, 
in which $\widetilde W^{0+}_1$ is
initialized at $x=L$ with approximate value $\tilde V_1$
\cite{BrZ,PZ,BHRZ}.
Thus, so long as the flow on the regular region $x\le L$
well-approximates the exact limiting flow as $v_+\to 0$, 
we can expect convergence of $D$ to $D^0$ based on the known
convergence of the numerical approximation scheme.
\end{remark}

\subsection{Convergence to $D^0$}\label{convergence}

As hinted by Remark \ref{num}, the rest of our analysis
is standard if not entirely routine.

\begin{lemma}\label{regconj}
On $x\le L$ for any fixed $L>0$, there exists a coordinate-change
$W=TZ$ conjugating \eqref{evans_ode} to the limiting equations
\eqref{limevans_ode}, $T=T(x,\lambda, v_+)$, satisfying a uniform bound
\begin{equation}\label{Tbd}
|T-Id|\le C(L)v_+
\end{equation}
for all $v_+> 0$ sufficiently small.
\end{lemma}

\begin{proof}
For $x\in (-\infty, 0]$, this is a consequence of the
{\it Convergence Lemma} of \cite{PZ}, a variation on
Proposition \ref{conjugation}, together with uniform
convergence of the profile and eigenvalue equations.
For $x\in [0,L]$, it is essentially continuous dependence;
more precisely, observing that
$ |A-A^0|\le C_1(L)v_+$ for $x\in [0,L]$, 
setting $S:=T-Id$, and writing the
homological equation expressing conjugacy of \eqref{evans_ode}
and \eqref{limevans_ode}, we obtain
$$
S'- (AS-SA^0)= (A-A^0),
$$
which, considered as an inhomogeneous linear matrix-valued equation, yields
an exponential growth bound 
\[
S(x)\le e^{Cx}(S(0)+ C^{-1}C_1(L)v_+)\]
for some $C>0$, giving the result.
\end{proof}

\begin{proof}[Proof of Theorem \ref{mainthm}]
Lemma \ref{regconj}, together with convergence as $v_+\to 0$
of the unstable subspace of $A_-$ to the unstable subspace
of $A^0_-$ at the same rate $O(v_+)$ (as follows by spectral separation
of the unstable eigenvalue of $A^0$ and standard matrix
perturbation theory) yields
\begin{equation}\label{Wbd}
|W_1^-(0,\lambda)- W_1^{0-}(0,\lambda)|\le C(L)v_+.
\end{equation}

Likewise, Lemma \ref{regconj} gives
\begin{equation}\label{tildeWbd}
\begin{aligned}
|\tilde W_1^+(0,\lambda)- \tilde W_1^{0+}(0,\lambda)|&\le
C(L)v_+
|\tilde W_1^+(0,\lambda)|\\
&\quad +
|S_0^{L \to 0}|
|\tilde W_1^+(L,\lambda)- \tilde W_1^{0+}(L,\lambda)|,
\end{aligned}
\end{equation}
where $S_0^{y\to x}$ denotes the solution operator of
the limiting adjoint eigenvalue equation $\tilde W'=-(A^0)^*\tilde W$.
Applying Proposition \ref{conjugation} to the limiting system, we obtain
$$
|S_0^{L\to 0}|\le C_2e^{-A^0_+ L}\le C_2L|\lambda|
$$
by direct computation of $e^{-A^0_+ L}$, where $C_2$ is independent of $L>0$.
Together with \eqref{wcon2} and \eqref{tildeWbd}, this gives
$$
|\tilde W_1^+(0,\lambda)- \tilde W_1^{0+}(0,\lambda)|\le
C(L)v_+
|\tilde W_1^+(0,\lambda)| + L|\lambda|C_2Ce^{-\eta L},
$$
hence, for $|\lambda|$ bounded,
\begin{equation}\label{lastbd}
\begin{aligned}
|\tilde W_1^+(0,\lambda)- \tilde W_1^{0+}(0,\lambda)|&\le
C_3(L)v_+ |\tilde W_1^{0+}(0,\lambda)| + LC_4e^{-\eta L}\\
&\le
C_5(L)v_+  + LC_4e^{-\eta L}.\\
\end{aligned}
\end{equation}
Taking first $L\to \infty$ and then $v_+\to 0$, we obtain
therefore convergence of $W^+_1(0,\lambda)$ and $\tilde W_1^+(0,\lambda)$ to
$W^{0+}_1(0,\lambda)$ and $\tilde W_1^{0+}(0,\lambda)$, yielding the result
by definitions \eqref{adjevans} and \eqref{duallimD}.
\end{proof}

%
\section{Numerical convergence}\label{quant}

Having established analytically convergence of
$D$ to $D^0$ as $M\to \infty$, 
we turn finally to numerics to obtain quantitative information
yielding a concrete stability threshold.
Specifically, for fixed $\gamma$, we compute the ``Rouch\'e bound''
$v_+$ at which the maximum relative error $|D-D^0|/|D^0|$ over the
semicircular contour $\partial \{\R \lambda \ge 0, \, |\lambda|\le 10\}$
around which we perform our winding number calculations
becomes $1/2$.
Recall that Rouch\'e's Theorem guarantees for relative error $<1$
that the winding number of $D$ is equal to the winding number of $D^0$,
which we have shown to be zero, hence we may conservatively conclude stability
for $v_+$ less than or equal to this bound, or $M$ greater than or
equal to the corresponding Mach number.
Computations are performed using the algorithm of \cite{BHRZ};
results are displayed in Table \ref{comparison3}.

More detailed results are displayed
for the monatomic gas case, $\gamma=1.66...$, 
in Table \ref{comparison2}.
Results are similar for other $\gamma\in [1,3]$,
as may be seen by comparing Figures \ref{zero2}, \ref{zero1a},
and \ref{zero3}.

\begin{table}
\begin{center}
\begin{tabular}{| c | c | c | c |}
\hline
  & & Relative & Mach\\
$\gamma$ & $v_+$ & Error& Number \\
\hline
3.0    &     1.27e-3 &   .5009 &     12765\\
2.5   &      1.36e-3 &   .5006 &     2423\\
2.0   &      1.49e-3 &   .5001 &     474\\
1.5  &     1.75e-3   &   .4999 &    95.5\\
1.0  &     2.8e-3   &   .4995 &     18.9\\
\hline
\end{tabular}
\bigskip
\caption{
Rouch\'e bounds for various $\gamma$.}
%
\label{comparison3}
\end{center}
\end{table}

\begin{table}
\begin{center}
\begin{tabular}{| c | c | c | c |}
\hline
& Mach & Relative & Absolute \\
$v_+$ & Number & Difference & Difference \\
\hline
1.0(-6) & 7.71(4) & 0.1221 & 0.0601\\
1.0(-5) & 1.13(4) & 0.1236 & 0.1445\\
1.0(-4) & 1.64(3) & 0.1487 & 0.4714\\
1.0(-3) & 2.44(2) & 0.4098 & 1.3464\\
1.0(-2) & 36.1 & 0.9046 & 2.8253\\
1.0(-1) & 5.50 & 1.2386 & 3.8688\\
\hline
\end{tabular}
\bigskip
\caption{
Maximum relative and absolute differences between $D$ and $D^0$,
for $\gamma=1.66...$ and $\lambda$ on the semicircle of radius $10$.}
\label{comparison2}
\end{center}
\end{table}

From the quantitative gap and conjugation estimates
given in Appendix \ref{quantlem}, from which follows
also a quantitative version of the Convergence Lemma
of \cite{PZ}, one could in principle establish quantitative
convergence rates for $|D-D^0|$, by tracking constants
carefully through the estimates of the previous sections.  
Indeed, one could do much better than the rather crude
bounds stated for the general case by taking into account
the eigenstructure of the actual matrices $A_\pm$, $A^0_\pm$ 
appearing in our analysis.
That is, there are contained in our analysis, as in the
study of \cite{BHRZ}, all of the ingredients
needed for a numerical proof.
Given the fundamental nature of the problem studied,
this would be a very interesting program to carry out.

\section{Discussion and open problems}\label{discussion}

Besides long-time stability, our results have application also to
existence of shock layers in
the small-viscosity limit, which likewise reduces
to the question of stability of the Evans function
\cite{R,GMWZ}.
Indeed, spectral stability has been a key missing piece
in several directions \cite{Z.2,Z.3}.
Our methods should have application also to spectral stability of
large-amplitude noncharacteristic boundary layers, completing the
investigations of \cite{SZ,GR,MeZ}.
It may be hoped that they will extend also to full gas dynamics
and multi-dimensions, two important directions for further investigation.
As discussed in the text, the problems of numerical proof and
of stability in the large-$\gamma$ limit are two other interesting
directions for further study.

More speculatively, our results suggest the possibility of
a large-variation version of the results obtained by quite
different methods in \cite{BiBr} on general
viscous solutions (including not only noninteracting shocks, but
shocks, rarefactions, and their interactions),
and, through the physical insight provided into the high-Mach
number limit, perhaps even a hint toward possible methods of analysis.
This would be an extremely interesting direction for further investigation.

\appendix

\section{Proofs of Preliminary Estimates}\label{basicproof}
\begin{proof}[Proof of Proposition \ref{profdecay}]
Existence and monotonicity follow trivially
by the fact that \eqref{profeq} is a scalar
first-order ODE with convex righthand side.
Exponential convergence as $x\to +\infty$ follows, for example,
by the computation
\begin{align*}
H(v, v_+) &= v\Big((v-1) -
\frac{(v_+-1)(v^{-\gamma}-1)}{v_+^{-\gamma}-1}\Big)\\
&= v\Big((v-v_+) + \Big(\frac{1-v_+}{1-v_+^{\gamma}}\Big)
\Big(\big(\frac{v_+}{v}\big)^\gamma -1\Big)\Big)\\
&= (v-v_+)\Big(v - \Big(\frac{1-v_+}{1-v_+^{\gamma}}\Big)
\Big(\frac{1 - \big(\frac{v_+}{v}\big)^{\gamma}}{1 -
\big(\frac{v_+}{v}\big)}\Big)\Big),\\
\end{align*}
yielding
\[
v- \gamma \le \frac{H(v,v_+)}{v-v_+}\le v-(1-v_+)
\]
by the elementary estimate $1\le \frac{1-x^\gamma}{1-x}\le \gamma$
for $0\le x\le 1$.  Convergence as $x\to -\infty$ follows by
a similar, but simpler computation; see \cite{BHRZ}.
\end{proof}

\begin{lemma}
The following identity holds for $\R \lambda \geq 0$:
\begin{align}
(\R(\lambda) + |\I (\lambda)|) & \ip \bV |u|^2 - \frac{1}{2}\ip \bV_x |u|^2 + \ip |u'|^2\notag\\
 &\leq \sqrt{2} \ip \frac{h(\bV)}{\bV^\gamma} |v'| |u| +  \ip \bV |u'||u|\label{id1}.
\end{align}
\end{lemma}

\begin{proof}
We multiply \eqref{ep:2} by $\bV {\bar u}$ and integrate along $x$.  This yields
\[
\lambda \ip \bV |u|^2 + \ip \bV u'\bar{u} + \ip |u'|^2 = \ip \frac{h(\bV)}{\bV^\gamma} v'\bar{u} .
\]
We get \eqref{id1} by taking the real and imaginary parts and adding them together, and noting that $|\R(z)| + |\I(z)| \leq \sqrt{2}|z|$.\qed
\end{proof}

\begin{lemma}
\label{kawashima}
The following identity holds for $\R \lambda \geq 0$:
%
%
\begin{equation}
\label{id3}
\ip |u'|^2 = 2\R(\lambda)^2\ip|v|^2 + \R(\lambda)\ip \frac{|v'|^2}{\bV} + \frac{1}{2} \ip \left[\frac{h(\bV)}{\bV^{\gamma+1}} + \frac{a\gamma}{\bV^{\gamma+1}} \right] |v'|^2
\end{equation}
\end{lemma}

\begin{proof}
We multiply \eqref{ep:2} by ${\bar v'}$ and integrate along $x$.  This yields
\[
\lambda \ip u\bar{v}' + \ip u'\bar{v}' - \ip \frac{h(\bV)}{\bV^{\gamma+1}}|v'|^2 = \ip \frac{1}{\bV}u''\bar{v}' = \ip \frac{1}{\bV}(\lambda v' + v''){\bar v'}.
\]
Using \eqref{ep:1} on the right-hand side, integrating by parts, and taking the real part gives
\[
\R \left[ \lambda \ip u\bar{v}' + \ip u'\bar{v}'\right] = \ip \left[\frac{h(\bV)}{\bV^{\gamma+1}} + \frac{\bV_x}{2 \bV^2} \right] |v'|^2 + \R(\lambda)\ip \frac{|v'|^2}{\bV}.
\]
The right hand side can be rewritten as
%
%
\begin{equation}
\label{id3_1}
\R \left[ \lambda \ip u\bar{v}' + \ip u'\bar{v}'\right] = \frac{1}{2} \ip \left[\frac{h(\bV)}{\bV^{\gamma+1}} + \frac{a\gamma}{\bV^{\gamma+1}} \right] |v'|^2 + \R(\lambda)\ip \frac{|v'|^2}{\bV}.
\end{equation}
Now we manipulate the left-hand side.  Note that
\begin{align*}
\lambda \ip u\bar{v}' + \ip u'\bar{v}' &= (\lambda+\bar{\lambda}) \ip u\bar{v}' - \ip u(\bar{\lambda}\bar{v}' + \bar{v}'')\\
&= -2\R(\lambda) \ip u' \bar{v} - \ip u \bar{u}''\\
&= -2\R(\lambda) \ip (\lambda v + v') \bar{v} + \ip |u'|^2.
\end{align*}
Hence, by taking the real part we get
\[
\R \left[ \lambda \ip u\bar{v}' + \ip u'\bar{v}'\right] = \ip |u'|^2 - 2\R(\lambda)^2 \ip |v|^2.
\]
This combines with \eqref{id3_1} to give \eqref{id3}.\qed
\end{proof}

\begin{lemma}\label{Hlem}
For $h(\bV)$ as in \eqref{f}, we have
\begin{equation}
\label{id2}
\sup_{\bV} \left| \frac{h(\bV)}{\bV^\gamma}\right| = \gamma 
\frac{1-v_+}{1-v_+^\gamma}
\leq \gamma.
\end{equation}
\end{lemma}

\begin{proof}
Defining
\begin{equation}\label{H}
g(\bV):=h(\bV)\bV^{-\gamma} = -\bV + a(\gamma-1)\bV^{-\gamma} + (a+1),
\end{equation}
we have $g'(\bV)= -1 -a\gamma(\gamma-1)\bV^{-\gamma-1}<0$ for $0<v_+\le \bV\le v_-= 1$,
hence the maximum of $g$ on $\bV\in [v_+,v_-]$ is achieved at
$\bV=v_+$.
Substituting \eqref{RH} into \eqref{H} and simplifying
yields \eqref{id2}.
\end{proof}

\begin{proof}[Proof of Proposition \ref{hf}]
Using Young's inequality twice on right-hand side of \eqref{id1} together with \eqref{id2}, we get
\begin{align*}
(\R&(\lambda) + |\I (\lambda)|) \ip \bV |u|^2 - \frac{1}{2}\ip \bV_x |u|^2 + \ip |u'|^2 \\
&\leq \sqrt{2} \ip \frac{h(\bV)}{\bV^\gamma} |v'| |u| +  \ip \bV |u'||u|\\
&\leq \theta \ip \frac{h(\bV)}{\bV^{\gamma+1}} |v'|^2 + \frac{(\sqrt{2})^2}{4\theta} \ip \frac{h(\bV)}{\bV^\gamma} \bV |u|^2 + \epsilon \ip \bV |u'|^2 + \frac{1}{4 \epsilon} \ip \bV |u|^2\\
&< \theta \ip \frac{h(\bV)}{\bV^{\gamma+1}} |v'|^2  + \epsilon \ip |u'|^2 + \left[\frac{\gamma}{2\theta} + \frac{1}{4 \epsilon}\right] \ip \bV |u|^2.
\end{align*}
Assuming that $0<\epsilon<1$ and $\theta = (1-\epsilon)/2$, this simplifies to
\begin{align*}
(\R(\lambda) + |\I (\lambda)|) & \ip \bV |u|^2 + (1-\epsilon) \ip |u'|^2 \\
&<\frac{1-\epsilon}{2} \ip \frac{h(\bV)}{\bV^{\gamma+1}} |v'|^2 +  \left[\frac{\gamma}{2\theta} + \frac{1}{4 \epsilon}\right] \ip \bV |u|^2.
\end{align*}
Applying \eqref{id3} yields
\[
(\R(\lambda) + |\I (\lambda)|) \ip \bV |u|^2  <  \left[\frac{\gamma}{1-\epsilon} + \frac{1}{4 \epsilon}\right]  \ip \bV |u|^2,
\]
or equivalently,
\[
(\R(\lambda) + |\I (\lambda)|) <  \frac{(4 \gamma-1)\epsilon - 1}{4\epsilon(1-\epsilon)}.
\]
Setting $\epsilon = 1/(2\sqrt{\gamma}+1)$ gives \eqref{hfbounds}.\qed
\end{proof}

\section{Nonvanishing of $D^0$}\label{nonvanishing}

As pointed out in Section \ref{phys},
the limiting eigenvalue system \eqref{limfirstorder},
\eqref{limevans_ode}, together with the limiting boundary
conditions derived in Section \ref{redEvans} may
be expressed equivalently as the integrated eigenvalue problem
\begin{subequations}\label{limitep}
\begin{align}
&\lambda v + v' - u' =0, \label{limitep:1}\\
&\lambda u + u' -  \frac{1-\bV}{\bV} v' = \frac{u''}{\bV}.\label{limitep:2}
\end{align}
\end{subequations}
corresponding to a pressureless gas, $\gamma=0$,
with special boundary conditions
\begin{equation}\label{PLBC}
(u,u',v,v')(-\infty)=(0,0,0,0),\:\:
(u,u',v,v')(+\infty)=(c,0,0,0).
\end{equation}

Motivated by this observation, we establish stability of the limiting
system by a Matsumura--Nishihara-type spectral energy estimate 
exactly analogous to that used to prove stability for
$\gamma=1$ in \cite{MN,BHRZ}.

\begin{proof}[Proof of Proposition \ref{redenergy}]
Multiplying \eqref{limitep:2} by $\bV \bar{u}/(1-\bV)$ and integrating on
some subinterval $[a.b]\subset\mathbb{R}$, we obtain
\[
\lambda \int^b_a \frac{\bV}{1-\bV}|u|^2 dx +  \int^b_a \frac{\bV}{1-\bV}
u'\bar{u} dx -  \int^b_a v'\bar{u} dx =  \int^b_a\frac{u'' \bar{u}}{1-\bV}
dx.
\]
Integrating the third and fourth terms by parts yields
\begin{align*}
\lambda \int^b_a \frac{\bV}{1-\bV}|u|^2 dx &+  \int^b_a \left[
\frac{\bV}{1-\bV} + \left( \frac{1}{1-\bV} \right)'\right] u'\bar{u} dx \\
&\quad+
\int^b_a \frac{|u'|^2}{1-\bV} dx
+  \int^b_a v (\overline{\lambda v + v'}) dx  \\
&= \left[ v \bar{u} + \frac{u'\bar{u}}{1-\bV} \right] \Big|^b_a.\\
\end{align*}
Taking the real part, we have

\begin{align}
&\R(\lambda) \int^b_a \left( \frac{\bV}{1-\bV}|u|^2 + |v|^2\right) dx + 
\int^b_a g(\bV) |u|^2 dx + \int^b_a \frac{|u'|^2}{1-\bV} dx\notag\\
&\quad  = \R \left[ v \bar{u} +
\frac{u'\bar{u}}{1-\bV} - \frac{1}{2}\left[\frac{\bV}{1-\bV} + \left(
\frac{1}{1-\bV} \right)'\right] |u|^2 - \frac{|v|^2 }{2} \right] \Big|^b_a,\label{MNen}
\end{align}
where
\[
g(\bV) =  -\frac{1}{2}\left[ \left(\frac{\bV}{1-\bV}\right)' + \left(
\frac{1}{1-\bV} \right)''\right].
\]
Note that
\[
\frac{d}{dx}\left(\frac{1}{1-\bV}\right) = - \frac{(1-\bV)'}{(1-\bV)2} =
\frac{\bV_x}{(1-\bV)2} = \frac{\bV(\bV-1)}{(1-\bV)2} =
-\frac{\bV}{1-\bV}.
\]
Thus, $g(\bV)\equiv 0$ and the third term on the right-hand side vanishes,
leaving
\begin{align*}
\R(\lambda) \int^b_a \left( \frac{\bV}{1-\bV}|u|^2 + |v|^2\right) dx &+
\int^b_a \frac{|u'|^2}{1-\bV} dx\\
&\quad  = \left[ \R(v \bar{u}) + \frac{\R(u'\bar{u})}{1-\bV} - \frac{|v|^2
}{2} \right] \Big|^b_a.
\end{align*}

We show next that the right-hand side goes to zero in the limit as 
$a\rightarrow-\infty$ and $b\rightarrow\infty$.  
By Proposition \ref{conjugation}, the behavior of $u$, $v$ near
$\pm \infty$ is governed by the limiting constant--coefficient
systems $W'=A^0_\pm(\lambda)W$, where $W=(u,v,v')^T$ 
and $A^0_\pm =A^0(\pm \infty, \lambda)$.
In particular, solutions $W$ asymptotic to $(1,0,0)$ at
$x=+\infty$ decay exponentially in $(u',v,v')$ and are bounded
in coordinate $u$ as $x\to +\infty$.
Observing that $1-\hat v\to 1$ as $x\to +\infty$, we thus see immediately
that the boundary contribution at $b$ vanishes as $b\to +\infty$.

The situation at $-\infty$ is 
more delicate, since the 
denominator $1-\hat v$ of the second term 
goes to zero at rate $e^{x}$ as $x\to -\infty$,
the rate of convergence of the limiting profile $\hat v$.
By inspection, the limiting coefficient matrix
\begin{equation}
\label{lim-}
A^0_-=
\begin{pmatrix}0 & \lambda & 1\\0 & 0 & 1\\ 
\lambda & \lambda  & 1-\lambda \end{pmatrix},
\end{equation}
has eigenvalues
$$
\mu= -\lambda, \, \frac{1 \pm \sqrt{1+4\lambda}}{2}, 
$$
hence for $\R \lambda \ge 0$ 
the unique decaying mode at $x=+\infty$ is the unstable eigenvector
corresponding to $\mu= \frac{1 + \sqrt{1+4\lambda}}{2}$,  
with growth rate
\[
\R\left( \frac{1 + \sqrt{1+4\lambda}}{2}\right) = \frac{1}{2} +
\frac{1}{2}\R\sqrt{1+4\lambda} > 1/2.
\]
Thus, $|u|, |u'|,|v'|,|v| \le Ce^{(1+ \epsilon)x/2}$ as $x\to -\infty$, 
$\epsilon >0$, and in particular 
$$
\Big|\frac{\R(u'\bar{u})}{1-\bV}\Big|\le Ce^{(1+\epsilon)x}/e^{x}
\le Ce^{\epsilon x}\to 0
$$
as $x\to -\infty$.
It follows that the boundary contribution at $x=a$ vanishes also
as $a\to -\infty$,
hence, in the limit as $a\to -\infty$, $b\to +\infty$,
\begin{equation}
\R(\lambda) \int^{+\infty}_{-\infty} \left( \frac{\bV}{1-\bV}|u|^2 + |v|^2\right) dx 
+ \int^{+\infty}_{-\infty} \frac{|u'|^2}{1-\bV} dx=0.
\end{equation}
But, for $\R \lambda\ge 0$, this implies $u'\equiv 0$,
or $u\equiv \hbox{\rm constant}$, which, by $u(-\infty)=0$,
implies $u\equiv 0$.
This reduces \eqref{limitep:1} to $v'=\lambda v$, yielding the
explicit solution $v=Ce^{\lambda x}$. By $v(\pm \infty)=0$,
therefore, $v\equiv 0$ for $\R \lambda\ge 0$.
It follows that there are no nontrivial solutions of \eqref{limitep},
\eqref{PLBC} for $\R \lambda\ge 0$.
\end{proof}

\begin{remark}\label{symmetrizers}
The above energy estimate is equivalent to multiplying
the system by the special symmetrizer 
$\begin{pmatrix} 1& 0 \\ 0 & \bV /(1-\bV) \end{pmatrix}$,
then taking the $L^2$ inner product with $(v,u)^T$.
The analog of the high-frequency estimates of Appendix \ref{basicproof} 
would be obtained using the alternative symmetrizer
$\begin{pmatrix} 1-\bV& 0 \\ 0 & \bV \end{pmatrix}$
optimized for its effect on second-order derivative term $u''/\hat v$.
This may clarify somewhat the strategy 
of the energy estimates used in \cite{MN,BHRZ}.
\end{remark}

\section{Quantitative conjugation estimates}\label{quantlem}
Consider a general first-order system
\begin{equation}\label{genode}
W'=A(x,\lambda)W.
\end{equation}

\begin{proposition}[Quantitative Gap Lemma \cite{GZ,ZH}]\label{easygap}
Let ${V}^+$ and ${\mu}^+$ be an eigenvector and
associated eigenvalue of $A_+(\lambda)$ and suppose that
there exist complementary generalized eigenprojections
(i.e., $A$-invariant projections) $P$ and $Q$ such that
\begin{equation} \label{assume}
\begin{aligned}
|Pe^{(A_+-{\mu}^+)x}|&\le C_1 e^{-\hat\eta x} \quad x\le 0,\\
|Qe^{(A_+-{\mu}^+)x}|&\le C_1 e^{-\hat\eta x} \quad x\ge 0,\\
|(A-A_+)(x)|&\le C_2 e^{-\eta x} \quad x\ge 0,\\
\end{aligned}
\end{equation}
with $0\le \hat \eta <\eta$.
Then, there exists a solution $W=e^{{\mu}^+ x}
{V}(x, \lambda)$ of \eqref{genode} with
\begin{equation}\label{Vbd}
\frac{|{V}(x,\lambda)- {V}^+(\lambda)|}
{|{V}^+(\lambda)|}
 \le \frac{C_1 C_2 e^{-\eta x}}{(\eta-\hat \eta) (1-\epsilon)} \quad 
\hbox{for $x\ge L$}
\end{equation}
provided $(\eta-\hat \eta)^{-1} C_1 C_2 e^{-\eta L} \le \epsilon$.  
\end{proposition}

\begin{proof}
Writing ${V}'=(A_+ - {\mu}^+)V 
+ (A - A_+)V$ and imposing the limiting behavior 
${V}(+\infty, \lambda)={V}^+$, 
we seek a solution in the form $V=TV$,
\[
\begin{aligned}
{TV}(x)&:={V}^+ -
\int_x^{+\infty} 
Pe^{(A_+ -{\mu}^+)(x-y)}(A-A_+)V(y) dy \\
&
\quad +
\int_L^{x} 
Qe^{(A_+ -{\mu}^+)(x-y)}(A-A_+)V(y) dy,\\
\end{aligned}
\]
from which the result follows by a straightforward Contraction Mapping 
argument, using \eqref{assume} to compute that
\[
\begin{aligned}
|{TV_1}-TV_2|(x)&=
\Big| -\int_x^{+\infty} 
Pe^{(A_+ -{\mu}^+)(x-y)}(A-A_+)(V_1-V_2)(y) dy \\
&
\quad +
\int_L^{x} 
Qe^{(A_+ -{\mu}^+)(x-y)}(A-A_+)(V_1-V_2)(y) dy \Big|\\
&\le 
C_1C_2\int_L^{+\infty}e^{-\hat \eta (x-y)}e^{-\eta y}dy
\|V_1-V_2\|_{L^\infty[L,+\infty)} \\
&=
\frac{C_1C_2 e^{-\hat\eta x}e^{-(\eta-\hat \eta)L}}
{\eta-\hat \eta} \|V_1-V_2\|_{L^\infty[L,+\infty)}, \\
\end{aligned}
\]
and thus
$\|TV_1-TV_2\|_{L^\infty[L,+\infty)} \le
\frac{C_1C_2 e^{-\eta L}}
{\eta-\hat \eta} \|V_1-V_2\|_{L^\infty[L,+\infty)}$.
\end{proof}

\begin{corollary}\label{qcor1}
Let ${V}^+$ and ${\mu}^+$ be an eigenvector and
associated eigenvalue of $A_+(\lambda)$,
where $A_+$ is $n\times n$ with at most $k$ stable eigenvalues and 
\begin{equation} \label{assume2}
\begin{aligned}
\max |(A_+-\mu)_{ij}|\le C_0;\qquad
|(A-A_+)(x)|&\le C_2 e^{-\eta x} \quad x\ge 0,\\
\end{aligned}
\end{equation}
$0<\hat \eta <\eta$.  Then, there exists a solution $W=e^{{\mu}^+ x}
{V}(x, \lambda)$ of \eqref{genode} with
\begin{equation}\label{Vbd2}
\frac{|{V}(x,\lambda)- {V}^+(\lambda)|}
{|{V}^+(\lambda)|}
 \le 
\frac{16n n! (C_0)^n C_2 e^{-\hat \eta x}}
{\delta^n(\eta-\hat \eta)(1-\epsilon)}
\quad
\hbox{for $x\ge L$},
\end{equation}
$\delta:= \frac{\eta-\hat \eta}{2k+2}$, provided 
$
\frac{16n n! (C_0)^n C_2 e^{-\hat \eta L}}
{\delta^n(\eta-\hat \eta)}
\le \epsilon$.
\end{corollary}

\begin{proof}
Without loss of generality, take $\mu\equiv 0$.
Dividing $[-\eta, -\hat \eta]$ into $k+1$ equal subintervals, we find
by the pigeonhole principle that at least one subinterval
contains the real part of no eigenvalue of $A_+$.  Denoting
the midpoint of this interval by $-\tilde \eta> \hat \eta$, we have
\begin{equation}\label{hatfacts}
\min |\R \sigma (A_+)- \tilde \eta| \ge \delta:= \frac{\eta-\hat \eta}{2k+2}
\, .
\end{equation}
Defining $P$ to be the total eigenprojection of $A_+$
associated with eigenvalues of real part greater than $\hat \eta$
and $Q$ the total eigenprojection associated with eigenvalues of 
real part less than $\hat \eta$, and estimating
$Pe^{A_+x}$, $Qe^{A_+x}$ using the  
the inverse Laplace transform representation
$$
e^{A_+x}=
\frac{1}{2\pi i}\oint_\Gamma e^{zx}(z-A_+)^{-1}dz,
$$
with $\Gamma$ chosen to be a rectangle of side $4nC_0$
centered about the real axis, with one vertical side
passing through $\R \lambda\equiv -\tilde \eta$ and
the other respectively lying respectively to the right and to the left,
and estimating 
$$
|(\lambda-A_+)^{-1}|\le n!C_0^{n-1}\delta^{-n}
$$
crudely by Kramer's rule, 
we obtain \eqref{assume} with
$C_1=16n n!C_0^{n}\delta^{-n},$ 
whence the result follows by Proposition \ref{easygap}.
\end{proof}

\begin{corollary}[Quantitative Conjugation Lemma]\label{qconj}
Proposition \ref{conjugation} holds with
$$
C(L,\hat \eta, \eta, \max |(M_+)_{ij}|, \dim M_+) =
\frac{16n n! (C_0)^n C_2 e^{-\hat \eta x}}
{\delta^n(\eta-\hat \eta)(1-\epsilon)},
$$
$n:=(\dim M_+)^2$, $k:=\frac{(\dim M_+)^2-\dim M}{2}$,
%
when
$
\frac{16n n! (C_0)^n C_2 e^{-\hat \eta L}}
{\delta^n(\eta-\hat \eta)}
\le \epsilon$.
\end{corollary}

\begin{proof}
Writing the homological equation expressing conjugacy of
variable- and constant-coefficient systems following
\cite{MeZ}, we have
$$
P'= M_+P-PM_+ + \Theta M.
$$
Considering this as an asymptotically constant-coefficient
system on the $n^2$-dimensional vector space of matrices $P$,
noting that the linear operator ${\mathcal M}_+P:=
M_+P-PM_+ $, as a Sylvester matrix, has at least $n$ zero eigenvalues
and equal numbers of stable and unstable eigenvalues, we see
that the number of its stable eigenvalues is not more
than $k:=\frac{n^2-n}{2}$, whence the result follows by
Corollary \ref{qcor1}.
\end{proof}

\def\cprime{$'$}

\end{document}